\documentclass[a4paper]{amsart}
\pdfoutput=1
\usepackage{etex}
\usepackage[utf8]{inputenc}
\usepackage[T1]{fontenc}
\usepackage{lmodern}
\usepackage{amssymb}
\usepackage{enumitem}
\setlist[enumerate,1]{label=\textup{(\arabic*)}}

\usepackage{mathtools,booktabs}

\usepackage{tikz-cd}

\usepackage[pdftitle={C*-algebras for partial product systems over N},
  pdfauthor={Ralf Meyer, Devarshi Mukherjee},
  pdfsubject={Mathematics}
]{hyperref}
\usepackage[lite]{amsrefs}

\renewcommand*{\PrintDOI}[1]{\href{http://dx.doi.org/\detokenize{#1}}{doi: \detokenize{#1}}}

\BibSpec{book}{%
  +{}  {\PrintPrimary}                {transition}
  +{,} { \textit}                     {title}
  +{.} { }                            {part}
  +{:} { \textit}                     {subtitle}
  +{,} { \PrintEdition}               {edition}
  +{}  { \PrintEditorsB}              {editor}
  +{,} { \PrintTranslatorsC}          {translator}
  +{,} { \PrintContributions}         {contribution}
  +{,} { }                            {series}
  +{,} { \voltext}                    {volume}
  +{,} { }                            {publisher}
  +{,} { }                            {organization}
  +{,} { }                            {address}
  +{,} { \PrintDateB}                 {date}
  +{,} { }                            {status}
  +{}  { \parenthesize}               {language}
  +{}  { \PrintTranslation}           {translation}
  +{;} { \PrintReprint}               {reprint}
  +{.} { }                            {note}
  +{.} {}                             {transition}
  +{} { \PrintDOI}                   {doi}
  +{} { available at \url}            {eprint}
  +{}  {\SentenceSpace \PrintReviews} {review}
}

\usepackage{microtype}

\numberwithin{equation}{section}
\theoremstyle{plain}
\newtheorem{theorem}[equation]{Theorem}
\newtheorem{lemma}[equation]{Lemma}
\newtheorem{claim}[equation]{Claim}
\newtheorem{proposition}[equation]{Proposition}

\theoremstyle{definition}
\newtheorem{definition}[equation]{Definition}
\theoremstyle{remark}
\newtheorem{remark}[equation]{Remark}
\newtheorem{example}[equation]{Example}

\newcommand*{\Hilm}[1][E]{\mathcal #1}
\newcommand{\idealin}{\mathrel{\triangleleft}} 

\newcommand*{\C}{\mathbb C}
\newcommand*{\Z}{\mathbb Z}
\newcommand*{\N}{\mathbb N}

\newcommand*{\T}{\mathbb T}
\newcommand*{\Bound}{\mathbb B}
\newcommand*{\Comp}{\mathbb K}
\newcommand*{\Mat}{\mathbb M}

\newcommand*{\Cont}{\mathrm C}
\newcommand*{\Contc}{\mathrm{C_c}}

\newcommand*{\id}{\mathrm{id}}

\newcommand*{\ima}{\mathrm i}
\newcommand*{\diff}{\mathrm d}
\newcommand*{\ev}{\mathrm{ev}}
\newcommand*{\Cst}{\mathrm C^*}
\newcommand*{\Star}{\texorpdfstring{$^*$\nb-}{*-}}

\newcommand*{\Toep}{\mathcal{T}} 
\newcommand*{\CP}{\mathcal{O}} 

\newcommand*{\nb}{\nobreakdash}  
\newcommand*{\alb}{\hspace{0pt}} 

\newcommand*{\blank}{\textup{\textvisiblespace}}

\mathtoolsset{mathic}
\DeclarePairedDelimiter{\norm}{\lVert}{\rVert}
\DeclarePairedDelimiter{\ket}{\lvert}{\rangle}
\DeclarePairedDelimiter{\bra}{\langle}{\rvert}
\DeclarePairedDelimiterX{\braket}[2]{\langle}{\rangle}{#1\,\delimsize\vert\,\mathopen{}#2}
\DeclarePairedDelimiterX{\BRAKET}[2]{\langle}{\rangle}{\!\delimsize\langle#1\,\delimsize\vert\,\mathopen{}#2\delimsize\rangle\!}
\DeclarePairedDelimiterX{\setgiven}[2]{\{}{\}}{#1\,{:}\,\mathopen{}#2}

\newcommand*{\congto}{\xrightarrow\sim}

\newcommand*{\defeq}{\mathrel{\vcentcolon=}}
\newcommand*{\into}{\rightarrowtail}
\newcommand*{\injto}{\hookrightarrow}
\newcommand*{\prto}{\twoheadrightarrow}

\begin{document}

\title{C*-Algebras for partial product systems over~$\N$}

\author{Ralf Meyer}
\email{rmeyer2@uni-goettingen.de}

\author{Devarshi Mukherjee}
\email{devarshi.mukherjee@mathematik.uni-goettingen.de}

\address{Mathematisches Institut\\
  Universit\"at G\"ottingen\\
  Bunsenstra\ss{}e 3--5\\
  37073 G\"ottingen\\
  Germany}

\begin{abstract}
  We define partial product systems over~\(\N\).
  They generalise product systems over~\(\N\)
  and Fell bundles over~\(\Z\).
  We define Toeplitz \(\Cst\)\nb-algebras
  and relative Cuntz--Pimsner algebras for them and show that the
  section \(\Cst\)-algebra
  of a Fell bundle over~\(\Z\)
  is a relative Cuntz--Pimsner algebra.  We describe the
  gauge-invariant ideals in the Toeplitz \(\Cst\)\nb-algebra.
\end{abstract}


\thanks{This article is based on the Master's Thesis of the second author.}

\maketitle

\section{Introduction}
\label{sec:intro}

The Cuntz--Pimsner algebras introduced by Pimsner
in~\cite{Pimsner:Generalizing_Cuntz-Krieger} were generalised, among
others, by Muhly and Solel~\cite{Muhly-Solel:Tensor} and by
Katsura~\cite{Katsura:Cstar_correspondences}.
Fowler~\cite{Fowler:Product_systems} generalised Pimsner's
construction to product systems.  A
self-correspondence~\(\Hilm\)
of a \(\Cst\)\nb-algebra~\(A\)
generates a product system over~\(\N\)
by taking \(\Hilm_n \defeq \Hilm^{\otimes_A n}\)
with the obvious multiplication maps
\(\mu_{n,m}\colon \Hilm_n \otimes_A \Hilm_m \to \Hilm_{n+m}\)
for \(n,m\in\N\).
Any product system over~\(\N\)
is isomorphic to one that is built from a \(\Cst\)\nb-correspondence
like this.  Another source of product systems over~\(\N\)
are Fell bundles over~\(\Z\)
(see~\cite{Exel:Partial_dynamical}).  They consist of Hilbert
bimodules~\((\Hilm_n)_{n\in\Z}\)
with involutions \(\Hilm_n \cong \Hilm_{-n}^*\),
\(x\mapsto x^*\),
and multiplication maps
\(\mu_{n,m}\colon \Hilm_n \otimes_A \Hilm_m \to \Hilm_{n+m}\),
now for all \(n,m\in\Z\).
Due to the involutions, the Hilbert bimodules~\(\Hilm_n\)
for \(n\in\N\)
with the multiplication maps~\(\mu_{n,m}\)
for \(n,m\in\N\)
suffice to recover the entire Fell bundle.  This data gives a
product system over~\(\N\)
if and only if the maps~\(\mu_{n,m}\)
are surjective and hence unitary for all \(n,m\in\N\).
Then the Fell bundle is called \emph{semi-saturated}.  In general, the
multiplication maps are only isometries of Hilbert bimodules, possibly
without adjoint.

The construction of a (relative) Cuntz--Pimsner algebra of a product
system splits into two steps.  The first builds a Fell bundle
over~\(\Z\),
the second takes the section \(\Cst\)\nb-algebra
of that Fell bundle.  This viewpoint is used
in~\cite{Meyer-Sehnem:Bicategorical_Pimsner} to interpret relative
Cuntz--Pismner algebras in bicategorical terms.  So there is a close
and important link between product systems over~\(\N\)
and Fell bundles over~\(\Z\).
This article describes a common generalisation for both, which we call
\emph{partial product systems}.

A partial product system over~\(\N\)
consists of a \(\Cst\)\nb-algebra~\(A\)
with \(A,A\)\nb-\alb{}correspondences
\(\Hilm_n\)
for all \(n\in\N\)
and isometries
\(\mu_{n,m}\colon \Hilm_n \otimes_A \Hilm_m \injto \Hilm_{n+m}\)
for all \(n,m\in\N\),
subject to several conditions.  The obvious conditions are that the
multiplication maps~\(\mu_{n,m}\)
be associative, that \(\Hilm_0=A\)
and that \(\mu_{0,n}\colon A\otimes_A \Hilm_n \injto \Hilm_n\)
and \(\mu_{n,0}\colon \Hilm_n\otimes_A A \injto \Hilm_n\)
be induced by the \(A\)\nb-bimodule structure on~\(\Hilm_n\)
for each \(n\in\N\).
Then we speak of a \emph{weak partial product system}.  Weak partial
product systems on von Neumann algebras have already been used in
the study of \(E_0\)\nb-semigroups,
where they are called ``superproduct systems'' (see
\cites{Margetts-Srinivasan:E0_II1, Bikram:CAR_flows_III}).

For a partial product system, we impose two more conditions to get a
well-behaved theory.  Their role is similar to the compact alignment
condition for product systems over quasi-lattice orders.  The
correspondences~\(\Hilm_n\)
in a weak partial product systems over~\(\N\)
are much more independent than in an ordinary product system,
and so the freeness of the monoid~\(\N\)
no longer helps.  This makes compact alignment and Nica covariance
relevant already over~\(\N\).
Our Toeplitz algebra is, in fact, an analogue of the Nica--Toeplitz
algebra.

Our first goal is to define the Toeplitz \(\Cst\)\nb-algebra
of a partial product system.  We define partial product systems so
that, on the one hand, this \(\Cst\)\nb-algebra
has a universal property for suitable representations of the partial
product system and, on the other hand, is generated concretely by an
analogue of the Fock representation.  The definition of a
representation has some obvious data and conditions and a non-obvious
condition needed to make the Toeplitz \(\Cst\)\nb-algebra
well-behaved.  We first discuss our definition of a representation.
Then we discuss the Fock representation.  Only then can we formulate
the remaining two conditions on partial product systems.  They say
simply that the Fock representation exists and is a representation.

\begin{definition}
  \label{def:weak_representation}
  Let~\(B\)
  be a \(\Cst\)\nb-algebra.
  A \emph{weak representation} of a weak partial product system
  \((A,\Hilm_n,\mu_{n,m})_{n,m\in\N}\)
  in~\(B\)
  consists of linear maps \(\omega_n\colon \Hilm_n \to B\)
  for all \(n\in\N\), such that
  \begin{enumerate}
  \item \label{def:weak_representation_1}%
    \(\omega_n(x)\cdot \omega_m(y) = \omega_{n+m}(\mu_{n,m}(x\otimes
    y))\) for all \(n,m\in\N\), \(x\in\Hilm_n\), \(y\in\Hilm_m\);
  \item \label{def:weak_representation_2}%
    \(\omega_n(x)^*\omega_n(y) = \omega_0(\braket{x}{y})\)
    for all \(n\in\N\), \(x,y\in\Hilm_n\);
  \end{enumerate}
  A weak representation is a \emph{representation} if, in addition,
  \begin{enumerate}[resume]
  \item \label{def:weak_representation_4}%
    \(\omega_n(\Hilm_n)^* \cdot \omega_m(\Hilm_m) \subseteq
    \omega_{m-n}(\Hilm_{m-n})\cdot B\)
    for all \(n,m\in\N\) with \(m>n>0\);
  \item \label{def:weak_representation_5}%
    \(\omega_n(\Hilm_n)^* \cdot \omega_m(\Hilm_m) \subseteq
    \omega_{n-m}(\Hilm_{n-m})^*\cdot B\)
    for all \(n,m\in\N\) with \(0<m<n\).
  \end{enumerate}
  By convention, \(X\cdot Y\)
  for two subspaces in a \(\Cst\)\nb-algebra~\(B\)
  always denotes the closed linear span of the products \(x\cdot y\)
  for \(x\in X\), \(y\in Y\).
\end{definition}

Conditions \ref{def:weak_representation_1}
and~\ref{def:weak_representation_2} for \(n=m=0\)
hold if and only if~\(\omega_0\) is a
\Star{}\alb{}homomorphism.  If \(n\in\N\),
then \ref{def:weak_representation_1} for \((n,0)\)
and \((0,n)\)
and \ref{def:weak_representation_2} for~\(n\)
say that~\(\omega_n\)
is a (Toeplitz) representation of the
\(\Cst\)\nb-correspondence~\(\Hilm_n\).

The definition of a partial product system uses the \emph{Fock
  representation}.  This should be a representation in
\(\Bound(\Hilm[F])\),
where~\(\Hilm[F]\)
is the Hilbert \(A\)\nb-module
direct sum \(\bigoplus_{n\in\N} \Hilm_n\).
For \(n,m\in\N\),
\(x\in\Hilm_n\),
define
\[
S_{n,m}(x)\colon \Hilm_m \to \Hilm_{n+m},\qquad
y\mapsto \mu_{n,m}(x\otimes y).
\]
This map is linear and \(\norm{S_{n,m}(x)} \le \norm{x}\).
The first assumption for a partial product system
asks~\(S_{n,m}(x)\)
to be adjointable for all \(n,m\in\N\),
\(x\in\Hilm_n\).
Then the operators~\(S_{n,m}(x)\)
for \(m\in\N\)
combine to an operator \(S_n(x)\in\Bound(\Hilm[F])\)
and the maps~\(S_n\)
form a weak representation.  The second assumption for a partial
product system is that they even form a representation, which we call
the \emph{Fock representation}:

\begin{definition}
  \label{def:partial_product_system}
  A \emph{partial product system} is a weak partial product system for
  which the Fock representation exists and is a representation.
\end{definition}

To understand this condition better, we reformulate
\ref{def:weak_representation_4} and~\ref{def:weak_representation_5} in
Definition~\ref{def:weak_representation} in case \(S_n(x)\)
is adjointable for all \(n\in\N\),
\(x\in\Hilm_n\).
Then \ref{def:weak_representation_4}
and~\ref{def:weak_representation_5} are equivalent to the first two
cases in the following equation:
\begin{equation}
  \label{eq:representation_condition}
  \omega_n(x)^* \omega_m(y) =
  \begin{cases}
    \omega_{m-n}(S_n(x)^* y)&\text{if }m>n,\\
    \omega_{n-m}(S_m(y)^* x)^*&\text{if }n>m,\\
    \omega_0(\braket{y}{x})^*&\text{if }n=m;
  \end{cases}
\end{equation}
here \(n,m\in\N\),
\(x\in\Hilm_n\),
\(y\in\Hilm_m\).
The cases \(n<m\)
and \(n>m\)
in~\eqref{eq:representation_condition} are equivalent to each other by
taking adjoints, and the case \(n=m\)
is condition~\ref{def:weak_representation_2} in
Definition~\ref{def:weak_representation}.  So a weak partial product
system is a partial product system if and only if the operators
\(S_n(x)\)
on~\(\Hilm[F]\)
are adjointable for all \(n\in\N\), \(x\in\Hilm_n\) and satisfy
\begin{equation}
  \label{eq:covariance_for_Fock}
  S_n(x)^* S_m(y) = S_{m-n}(S_n(x)^* y)
\end{equation}
for all \(n,m\in\N\),
\(x\in\Hilm_n\),
\(y\in\Hilm_m\)
with \(m>n\).
And a weak representation of a partial product system is a
representation if and only if
\(\omega_n(x)^* \omega_m(y) = \omega_{m-n}(S_n(x)^* y)\)
for all \(n,m\in\N\),
\(x\in\Hilm_n\),
\(y\in\Hilm_m\)
with \(m>n\).
The formulation of the extra conditions in
Definition~\ref{def:weak_representation} is inspired by a similar
treatment of the Cuntz--Pimsner covariance condition of a proper
product system in~\cite{Albandik-Meyer:Colimits} and makes sense
without the adjointability of~\(S_n(x)\).
The reformulation in~\eqref{eq:representation_condition} guarantees
that there is a universal \(\Cst\)\nb-algebra
for representations of our partial product system.

We describe an example of a weak partial product system
where~\eqref{eq:covariance_for_Fock} fails.  Its definition uses
correspondences based on graphs.  A graph with vertex and edge sets
\(V\)
and~\(E\)
gives a \(\Cst\)\nb-correspondence~\(\Cst(E)\)
over~\(\Cont_0(V)\).
To define a weak partial product system using graphs, we need a common
vertex set~\(V\),
graphs \(\Gamma_n=(V,E_n,r_n,s_n)\)
for all \(n\in\N\),
and associative, injective multiplication maps
\(\mu_{n,m}\colon E_n \times_{s,r} E_m \injto E_{n,m}\),
where \(E_0=V\)
and \(r_0,s_0\)
are the identity map, and \(\mu_{0,n}\)
and~\(\mu_{n,0}\)
are the canonical maps.  We study when~\(\mu_{n,m}\)
induces an isometry
\(\Cst(E_n)\otimes_{\Cont_0(V)} \Cst(E_m) \injto \Cst(E_{n+m})\)
and when these isometries form a partial product system.  These
conditions are rather restrictive.  As it turns out, the category with
object set~\(V\)
and arrow set \(\bigsqcup_{n\in\N} E_n\)
must be the path category of an ordinary graph.  The only variation is
that the grading is not the standard one, that is, elements of~\(E_n\)
need not be paths of length~\(n\).

The following theorem generalises an important feature of Pimsner's
Toeplitz algebras.  Our definitions above are arranged so as to make
it true.

\begin{theorem}
  \label{the:Toeplitz_through_Fock}
  Let \(\Hilm = (A,\Hilm_n,\mu_{n,m})_{n,m\in\N}\)
  be a partial product system.  Let \((S_n)_{n\in\N}\)
  be its Fock representation.  The closed linear span of
  \(S_n(\Hilm_n)S_m(\Hilm_m)^*\)
  for \(m,n\in\N\)
  is a \(\Cst\)\nb-subalgebra
  \(\Toep\)
  of~\(\Bound(\Hilm[F])\),
  and the maps~\(S_n\)
  form a representation~\(\bar{\omega}_n\)
  of~\(\Hilm\)
  in~\(\Toep\).
  This representation in~\(\Toep\)
  is universal: for any representation
  \((\omega_n\colon \Hilm_n \to B)_{n\in\N}\)
  of~\(\Hilm\)
  there is a unique \Star{}homomorphism \(\varrho\colon \Toep \to B\)
  with \(\omega_n = \varrho\circ \bar{\omega}_n\) for all \(n\in\N\).
\end{theorem}

We describe a gauge action of the circle group~\(\T\)
on~\(\Toep\)
and prove a gauge-equivariant uniqueness theorem for~\(\Toep\).
This allows us to prove Theorem~\ref{the:Toeplitz_through_Fock}.  We
also describe the fixed-point subalgebra of the gauge action
explicitly as an inductive limit \(\Cst\)\nb-algebra.
For \(m,n\in\N\),
there is a unique linear map
\[
\Theta_{m,n}\colon \Comp(\Hilm_m,\Hilm_n) \to \Toep
\]
with
\[
  \Theta_{m,n}(\ket{x}\bra{y}) =
  \bar{\omega}_n(x) \bar{\omega}_m(y)^*
\]
for all \(x\in\Hilm_n\),
\(y\in\Hilm_m\).
These maps are injective and their images are linearly independent
subspaces of~\(\Toep\).
Their direct sum over all \(m,n\in\N\)
is a dense \Star{}subalgebra in~\(\Toep\)
because of~\eqref{eq:representation_condition}.

We classify gauge-invariant ideals \(H\idealin\Toep\)
in the Toeplitz \(\Cst\)\nb-algebra
by two ideals in~\(A\),
namely, the kernel \(A\cap H\)
and the covariance ideal, consisting of all \(a\in A\)
that are equal modulo~\(H\)
to an element in the closed linear span of \(\Comp(\Hilm_i)\)
for \(i\ge1\).
Theorem~\ref{the:ideal_Toeplitz} describes~\(H\)
through its kernel and covariance ideal.  The kernel~\(I\)
is always an invariant ideal, and any invariant ideal may occur.  We
do not know, in general, which covariance ideals are possible.  We
also define what it means for a representation to be covariant on an
ideal.  If an ideal~\(J\)
is the covariance ideal of some representation, then there is a
universal \(\Cst\)\nb-algebra
\(\CP(\Hilm,J)\)
for representations of~\(\Hilm\)
that are covariant on~\(J\).
We characterise when the canonical representation of~\(\Hilm\)
in~\(\CP(\Hilm,J)\)
is faithful: this happens if and only if \(J\subseteq K^\bot\)
for \(K \defeq \bigcap_{n=1}^\infty \ker(\vartheta_0^n)\),
where \(\vartheta_0^n\colon A\to \Bound(\Hilm_n)\)
is the left action in the correspondence~\(\Hilm_n\).
This allows us to define an analogue of Katsura's \(\Cst\)\nb-algebra
for partial product systems.  We also define an analogue of Pimsner's
\(\Cst\)\nb-algebra
as the quotient of~\(\Toep\)
by the closure of the finite block matrices in the Fock
representation.  For global product systems, we show that the Katsura
and Pimsner \(\Cst\)\nb-algebras
defined here agree with those previously constructed.  We show that
the Katsura \(\Cst\)\nb-algebra
of a Fell bundle over~\(\Z\),
restricted to a partial product system over~\(\N\),
is the section \(\Cst\)\nb-algebra
of the Fell bundle.  Using this, we characterise when a partial
product system over~\(\N\)
is the restriction of a Fell bundle over~\(\Z\).
For the partial product systems built from graphs, we show that the
Katsura algebra is the graph \(\Cst\)\nb-algebra.

\section{Weak partial product systems and representations}
\label{sec:weak}

In this section, we recall some basic notions and study
general properties of weak partial product systems and their weak
representations.  We examine when a weak Fock representation
exists, that is, when the operators~\(S_n(x)\)
on the Fock module mentioned in the introduction are adjointable.  We
illustrate our theory with \(\Cst\)\nb-correspondences
built from graphs.

Let \(A\),
\(B\),
\(C\)
be \(\Cst\)\nb-algebras.
An \emph{\(A,B\)\nb-correspondence}
is a right Hilbert module~\(\Hilm\)
over~\(B\)
with a \Star{}homomorphism \(\vartheta\colon A \to \Bound(\Hilm)\)
satisfying
\(\braket{\vartheta(a)x}{y}_B = \braket{x}{\vartheta(a)^* y}_B\)
for all \(a\in A\),
\(x,y \in \Hilm\).
We often write~\(a x\)
instead of \(\vartheta(a)(x)\).
\emph{We do not require the left action of~\(A\)
  on~\(\Hilm\)
  to be nondegenerate, and we allow representations of
  \(\Cst\)\nb-algebras to be degenerate throughout this article.}

Let \(\Hilm\)
and~\(\Hilm[F]\)
be an \(A,B\)-
and a \(B,C\)-correspondence.
We equip the algebraic tensor product \(\Hilm\odot \Hilm[F]\)
with the obvious \(A,C\)-bimodule
structure \(a(x\otimes y)c=ax\otimes yc\)
for \(a\in A\),
\(c\in C\),
\(x \in \Hilm\),
\(y\in\Hilm[F]\), and with the \(C\)\nb-valued inner product
\begin{equation}
  \label{IP}
  \braket{x_1 \otimes y_1}{x_2 \otimes y_2}_C
  \defeq \braket{y_1}{\braket{x_1}{x_2}_A \cdot y_2}_B
\end{equation}
for \(x_1, x_2 \in \Hilm\)
and \(y_1, y_2 \in \Hilm[F]\).
The Hausdorff completion of \(\Hilm\odot \Hilm[F]\)
for this inner product is an \(A,C\)-correspondence
denoted by \(\Hilm\otimes_B \Hilm[F]\)
(see~\cite{Lance:Hilbert_modules}).  The \emph{identity
  correspondence}~\(\Hilm[A]\)
on~\(A\)
for a \(\Cst\)\nb-algebra~\(A\)
is~\(A\)
viewed as an \(A,A\)\nb-correspondence
using the obvious bimodule structure and the inner product
\(\braket{a}{b}_A \defeq a^* b\).

\begin{lemma}
  \label{ass}
  Let \({}_A\Hilm_B\),
  \({}_B\Hilm[F]_C\)
  and \({}_C\mathcal{G}_D\)
  be \(\Cst\)\nb-correspondences
  between the indicated \(\Cst\)\nb-algebras
  \(A,B,C,D\).  There are canonical isomorphisms of correspondences
  \begin{alignat*}{2}
    (\Hilm\otimes_B \Hilm[F]) \otimes_C \mathcal{G}
    &\cong \Hilm\otimes_B (\Hilm[F] \otimes_C \mathcal{G}),&\qquad
    (x\otimes y) \otimes z &\mapsto x \otimes (y \otimes z),\\
    \Hilm[A]\otimes_A \Hilm &\cong A\cdot \Hilm \subseteq \Hilm,&\qquad
    a \otimes x &\mapsto a x,\\
    \Hilm\otimes_B \Hilm[B] &\cong \Hilm,&\qquad
    x \otimes b &\mapsto x b.
  \end{alignat*}
\end{lemma}

We usually omit parentheses in tensor products and the associator
isomorphism in Lemma~\ref{ass} to reduce the size of our diagrams.
They are canonical enough that this cannot cause confusion.

\begin{definition}
  \label{def:weak_pps}
  A \emph{weak partial product system} over~\(\N\)
  consists of
  \begin{itemize}
  \item a \(\Cst\)\nb-algebra~\(A\),
  \item \(A,A\)-correspondences \(\Hilm_n\) for \(n\in\N_{\ge1}\),
  \item isometric bimodule maps
    \(\mu_{n,m}\colon \Hilm_n \otimes_A \Hilm_m \injto \Hilm_{n+m}\)
    for \(n,m\in\N_{\ge1}\);
  \end{itemize}
  such that the following diagrams commute for all
  \(n,m,l\in \N_{\ge1}\) (``associativity''):
  \begin{equation}
    \label{assoeqn}
    \begin{tikzcd}[column sep=huge]
      \Hilm_n \otimes_A \Hilm_m \otimes_A \Hilm_l
      \arrow[r, hookrightarrow, "1_{\Hilm_n}\otimes_A \mu_{m,l}"]
      \arrow[d, hookrightarrow, "\mu_{n,m}\otimes_A 1_{\Hilm_l}"'] &
      \Hilm_n \otimes_A \Hilm_{m+l}
      \arrow[d, hookrightarrow, "\mu_{n,m+l}"] \\
      \Hilm_{n+m} \otimes_A \Hilm_l
      \arrow[r, hookrightarrow, "\mu_{n+m,l}"] &
      \Hilm_{n+m+l}
    \end{tikzcd}
  \end{equation}
  (Being isometric means that
  \(\braket{\iota(x)}{\iota(y)} = \braket{x}{y}\)
  for all \(x,y\in\Hilm_n \otimes_A \Hilm_m\).)
  Let \(\Hilm_0\defeq A\)
  and let
  \(\mu_{0,m}\colon \Hilm[A] \otimes_A \Hilm_m \injto \Hilm_m\)
  and \(\mu_{m,0}\colon \Hilm_m\otimes_A \Hilm[A] \injto \Hilm_m\)
  be the canonical isometries from Lemma~\ref{ass}.  Then the
  diagram~\eqref{assoeqn} commute for all \(n,m,l\in \N\).
\end{definition}

Let \((A,\Hilm_n,\mu_{n,m})_{n,m\in\N}\)
be a weak partial product system.  Let~\((\omega_n)_{n\in\N}\)
be a weak representation of it in a \(\Cst\)\nb-algebra~\(B\)
as in Definition~\ref{def:weak_representation}.  That is,
\(\omega_n \colon \Hilm_n \to B\)
for \(n\in\N\)
are linear maps satisfying the conditions
\ref{def:weak_representation_1} and~\ref{def:weak_representation_2} in
Definition~\ref{def:weak_representation}.  Let \(m,n\in\N\).
By definition, \(\Comp(\Hilm_m, \Hilm_n)\)
is the closed linear span in \(\Bound(\Hilm_m,\Hilm_n)\)
of \(\ket{x} \bra{y}\)
for \(x\in \Hilm_n\),
\(y \in \Hilm_m\),
where \((\ket{x} \bra{y})(z)\defeq x\braket{y}{z}_A\)
for all \(z\in\Hilm_m\).
There is a unique map
\(\Theta_{m,n}\colon \Comp(\Hilm_m,\Hilm_n) \to B\) with
\begin{equation}
  \label{eq:def_Theta}
  \Theta_{m,n}(\ket{x} \bra{y}) = \omega_n(x) \omega_m(y)^*
\end{equation}
for all \(x\in\Hilm_n\),
\(y \in \Hilm_m\);
this follows from \cite{Pimsner:Generalizing_Cuntz-Krieger}*{Lemma
  3.2} applied to the representation
\(\omega_m\oplus\omega_n\colon \Hilm_m\oplus\Hilm_n \to \Mat_2(B)\)
of \(\Hilm_m\oplus\Hilm_n\),
by viewing \(\Comp(\Hilm_m,\Hilm_n)\)
as an off-diagonal corner in \(\Comp(\Hilm_m\oplus \Hilm_n)\).
These maps are compatible with the multiplication maps and adjoints,
that is,
\begin{equation}
  \label{eq:Theta_multiplicative}
  \Theta_{m,n}(S)\cdot \Theta_{l,m}(T) = \Theta_{l,n}(S\cdot T),\qquad
  \Theta_{n,m}(S^*) = \Theta_{m,n}(S)^*
\end{equation}
for all \(S\in\Comp(\Hilm_m,\Hilm_n)\),
\(T\in\Comp(\Hilm_l,\Hilm_m)\);
this follows from the case of rank-one operators, which in turn
follows easily from the conditions in
Definition~\ref{def:weak_representation}:
\[
\omega_n(x) \omega_m(y)^* \omega_m(z) \omega_l(w)^*
= \omega_n(x) \omega_0(\braket{y}{z}) \omega_l(w)^*
= \omega_n(x\cdot\braket{y}{z}) \omega_l(w)^*.
\]
In particular,
\[
\Theta_{n,0}(x) = \omega_n(x),\qquad
\Theta_{0,n}(x^*) = \omega_n(x)^*
\]
for all \(x\in \Hilm_n \cong \Comp(A,\Hilm_n)\).
So \eqref{eq:Theta_multiplicative} implies
\begin{equation}
  \label{eq:Theta_multiplicative_0}
  \Theta_{m,n}(T)\omega_m(x) = \omega_n(T(x))
\end{equation}
for all \(m,n\in\N\), \(T\in\Comp(\Hilm_m,\Hilm_n)\),
\(x\in\Hilm_m\).

The maps~\(\Theta_{m,n}\)
map \(\Comp(\Hilm_m,\Hilm_n)\)
into the \(\Cst\)\nb-subalgebra
of~\(B\)
that is generated by~\(\omega_n(\Hilm_n)\).
The closed linear span of \(\Theta_{m,n}(\Comp(\Hilm_m,\Hilm_n))\)
for \(m,n\in\N\)
need not be an algebra.  We will impose more relations to arrange for
this later.

\begin{definition}
  \label{def:invariant_ideal}
  An ideal \(J\idealin A\)
  is \emph{invariant} with respect to a weak partial product system
  \(\Hilm = (A,\Hilm_n,\mu_{n,m})_{n,m\in\N}\)
  if \(J\cdot \Hilm_n \subseteq \Hilm_n\cdot J\) for all \(n\in\N\).
\end{definition}

\begin{lemma}
  \label{lem:faithful_A_Hilm_Comp}
  Let~\((\omega_n)_{n\in\N}\)
  be a weak representation of a weak partial product system~\(\Hilm\)
  in a \(\Cst\)\nb-algebra~\(B\).
  The ideal \(I\defeq \ker\omega_0\) is invariant and
  \[
  \ker \omega_n = \Hilm_n\cdot I\subseteq \Hilm_n,\qquad
  \ker \Theta_{m,n} = \Comp(\Hilm_m,\Hilm_n \cdot I)
  \subseteq \Comp(\Hilm_m,\Hilm_n).
  \]
\end{lemma}

\begin{proof}
  Let \(x\in\Hilm_n\).
  We have \(\omega_n(x)=0\)
  if and only if
  \(0 = \omega_n(x)^* \omega_n(x) = \omega_0(\braket{x}{x})\),
  if and only if \(\braket{x}{x}\in I\).
  The latter is equivalent to \(x\in\Hilm_n\cdot I\).
  So \(\ker \omega_n = \Hilm_n\cdot I\).
  This implies that~\(I\)
  is invariant because \(I\cdot\Hilm_n\subseteq \ker\omega_n\).
  Let \(T\in\Comp(\Hilm_m,\Hilm_n)\).
  Then \(T\in\Comp(\Hilm_m,\Hilm_n\cdot I)\)
  if and only if \(T(x)\in \Hilm_n\cdot I\)
  for all \(x\in\Hilm_m\).
  By the first part, this is equivalent to
  \(\omega_n(T(x)) = \Theta_{m,n}(T)\omega_m(x)=0\)
  for all \(x\in\Hilm_m\).
  This follows if \(\Theta_{m,n}(T)=0\).
  Conversely, if \(\Theta_{m,n}(T)\omega_m(x)=0\)
  for all \(x\in\Hilm_m\),
  then \(\Theta_{m,n}(T)\cdot \Theta_{n,m}(S)=0\)
  for all \(S\in\Comp(\Hilm_n,\Hilm_m)\).
  Taking \(S=T^*\),
  this implies \(\Theta_{m,n}(T)=0\)
  because \(\Theta_{n,m}(T^*) = \Theta_{m,n}(T)^*\).
\end{proof}

Next we seek an analogue of the Fock representation.  We want this to
exist because it is used by
Pimsner~\cite{Pimsner:Generalizing_Cuntz-Krieger} to define the
Toeplitz \(\Cst\)\nb-algebra.
The Fock representation should be a (weak) representation on the
Hilbert \(A\)\nb-module
\(\Hilm[F] \defeq \bigoplus_{n=0}^\infty \Hilm_n\),
which we call the \emph{Fock module} of~\(\Hilm\).
Fix \(n\in\N\)
and \(x\in\Hilm_n\).
In the Fock representation, \(x\)
should act on the summand~\(\Hilm_m\) by the operator
\[
S_{n,m}(x)\colon \Hilm_m \to \Hilm_{n+m},\qquad
y\mapsto x\cdot y \defeq \mu_{n,m}(x\otimes y).
\]
More precisely, \(S_{n,m}(x)\)
is the composite of the isometry~\(\mu_{n,m}\)
with the creation operator
\(\Hilm_m \to \Hilm_n\otimes_A \Hilm_m\),
\(y\mapsto x\otimes y\),
which is adjointable.
So the operator above is a well defined bounded linear map.  It is
always adjointable for \(n=0\),
but not for \(n>0\).
Sufficient conditions for this are the following:
\begin{enumerate}
\item \label{cond:mu_nm_adjointable}%
  if the isometries~\(\mu_{n,m}\) are adjointable;
\item if the correspondence~\(\Hilm_n\)
  is proper, that is, \(A\)
  acts by compact operators on~\(\Hilm_n\):
  then the creation operator \(\Hilm_m \to \Hilm_n\otimes_A \Hilm_m\)
  is compact, and then so is~\(S_{n,m}(x)\)
  because \(\Comp(\Hilm[F]) \subseteq \Comp(\Hilm)\)
  if \(\Hilm[F]\subseteq \Hilm\)
  is a Hilbert submodule in a Hilbert module;
\item if~\((\Hilm_n)_{n\in\N}\)
  comes from a Fell bundle over~\(\Z\):
  then the left multiplication map \(\Hilm_{n+m} \to \Hilm_m\),
  \(y\mapsto x^*\cdot y\),
  for \(x^*\in\Hilm_{-n}\)
  is adjoint to~\(S_{n,m}(x)\).
\end{enumerate}
Condition~\ref{cond:mu_nm_adjointable} contains product systems in
the usual sense, where each~\(\mu_{n,m}\)
is unitary and hence adjointable.

The adjointability of the operators~\(S_{n,m}(x)\)
is one of the requirements for a partial product system.  In other
words, we require that for all \(x\in \Hilm_n\),
\(t\in \Hilm_{n+m}\),
there is \(z\in\Hilm_m\),
necessarily unique, with \(\braket{t}{x\cdot y}_A = \braket{z}{y}_A\)
for all \(y\in \Hilm_m\).
If the maps \(S_{n,m}(x)\)
are adjointable, then so is the \emph{creation operator}
\(S_n(x)\defeq \sum_{m\in\N} S_{n,m}(x)\)
on the Fock module~\(\Hilm[F]\);
we call the adjoint \(S_n(x)^*\)
an \emph{annihilation operator}.  It is easy to see that the
maps~\(S_n(x)\)
form a weak representation of our weak partial product system in
\(\Bound(\Hilm[F])\).
That is, \(S_n(x) S_m(y) = S_{n+m}(x\cdot y)\)
for \(x\in\Hilm_n\),
\(y\in\Hilm_m\),
\(n,m\in\N_{\ge1}\),
and \(S_n(x)^* S_n(y) = S_0(\braket{x}{y}_A)\)
for \(x,y\in\Hilm_n\),
\(n,m\in\N_{\ge1}\),
\(n\in\N_{\ge1}\).

\subsection{Correspondences associated to graphs}
\label{sec:graph}

A (directed) graph is given by countable discrete sets \(V\)
and~\(E\)
of vertices and edges and maps \((r,s)\colon E\rightrightarrows V\)
sending an edge to its range and source.  It yields a
\(\Cont_0(V),\Cont_0(V)\)-correspondence
by completing \(\Contc(E)\) in the \(\Cont_0(V)\)-valued inner product
\[
\braket{x}{y}(v) \defeq
\sum_{e \in s^{-1}(v)} \overline{x(e)} y(e)
\]
for \(x,y \in \Contc(E)\) and \(v\in V\) or, equivalently,
\[
\braket{\delta_x}{\delta_y}_A
=
\begin{cases}
  \delta_{s_n(x)} &\quad\text{if } x=y,\\
  0 &\quad\text{otherwise;}
\end{cases}
\]
the
\(\Cont_0(V)\)-bimodule structure is given by
\[
(x\cdot a)(e) \defeq x(e)\cdot a(s(e)),\qquad
(a\cdot x)(e) \defeq a(r(e))\cdot x(e)
\]
for \(e\in E\),
\(a\in \Cont_0(V)\),
\(x\in \Contc(E)\).
This is a \(\Cst\)\nb-correspondence,
and its Cuntz--Pimsner algebra (as modified by Katsura) is the graph
\(\Cst\)\nb-algebra of our graph (see~\cite{Raeburn:Graph_algebras}).

Now consider two graphs with the same vertex set~\(V\),
with sets of edges \(E_1\)
and~\(E_2\)
and range and source maps \((r_i,s_i)\colon E_i \rightrightarrows V\)
for \(i=1,2\).
As above, we build two \(\Cont_0(V),\Cont_0(V)\)-correspondences
\(\Hilm_1\)
and~\(\Hilm_2\).
The composite correspondence \(\Hilm_1 \otimes_{\Cont_0(V)} \Hilm_2\)
is associated to the graph with edge set
\(E \defeq E_1 \times_{s,r} E_2\)
and \(r,s\colon E \rightrightarrows V\)
defined by \(r(f,g)\defeq r_1(f)\) and \(s(f,g) \defeq s_2(g)\).

Now let \(r_n,s_n\colon E_n\rightrightarrows V\)
for \(n\in\N_{>0}\)
be graphs with the same vertex set~\(V\).
Let \(\tilde\mu_{n,m}\colon E_n \times_{s,r} E_m \to E_{n+m}\)
be \emph{injective} maps for all \(n,m\in\N\),
which we write multiplicatively as
\(x\cdot y \defeq \tilde\mu_{n,m}(x,y)\)
for \(x\in E_n\),
\(y\in E_m\)
with \(s_n(x)=r_m(y)\).
Assume that \(r_{n+m}(x\cdot y) = r_n(x)\)
and \(s_{n+m}(x\cdot y) = s_m(y)\)
and that these multiplication maps are associative.  Let
\(E_0=V\)
and \(r_0=s_0=\id_V\)
and let the multiplication maps
\(\tilde\mu_{0,m}\colon V\times_{s,r} E_m \to E_m\)
and \(\tilde\mu_{m,0}\colon E_m\times_{s,r} V \to E_m\)
be the obvious maps \((r(x),x)\mapsto x\),
\((x,s(x))\mapsto x\).
We are going to build a weak partial product system out of this data.
The construction will also show that the assumptions above are
necessary to get a weak partial product system.

Let \(A\defeq \Cont_0(V)\)
and let~\(\Hilm_n\)
be the \(A,A\)-correspondence
associated to the graph~\(E_n\)
as above.  The characteristic functions \((\delta_x)_{x\in E_n}\)
form a basis in~\(\Contc(E_n)\),
which is dense in~\(\Hilm_n\).
So \((\delta_x\otimes \delta_y)_{x\in E_n,y\in E_m}\)
is a basis in \(\Contc(E_n) \odot \Contc(E_m)\),
which maps to a dense subset in \(\Hilm_n \otimes_A \Hilm_m\).
For \(x\in E_n\), \(y\in E_m\), define
\[
\mu_{n,m}(\delta_x \otimes \delta_y) \defeq
\begin{cases}
  \delta_{x\cdot y}&\text{if }s_n(x)= r_m(y),\\
  0&\text{otherwise.}
\end{cases}
\]

\begin{lemma}
  The map
  \(\mu_{n,m}\colon \Contc(E_n) \odot \Contc(E_m) \to
  \Contc(E_{n+m})\)
  is isometric for the \(\Cont_0(V)\)-valued
  inner product and extends uniquely to an isometric bimodule map
  \(\Hilm_n \otimes_A \Hilm_m \injto \Hilm_{n+m}\).
\end{lemma}

\begin{proof}
  Let \(x,x'\in E_n\),
  \(y,y'\in E_m\).
  We write \(\delta_{a=b}\)
  for the Kronecker delta to distinguish it more clearly from our
  characteristic functions.  We declare that \(\delta_{x\cdot y}=0\)
  if \(s_n(x)\neq r_m(y)\),
  so that \(\mu_{n,m}(\delta_x\otimes \delta_y) = \delta_{x\cdot y}\)
  holds for all \(x\in E_n\),
  \(y\in E_m\).  On the one hand, we compute
  \begin{multline*}
    \braket[\big]{\mu_{n,m}(\delta_x\otimes \delta_y)}
    {\mu_{n,m}(\delta_{x'}\otimes \delta_{y'})}(v)
    = \braket{\delta_{x\cdot y}}{\delta_{x'\cdot y'}}(v)
    \\= \delta_{s(x)=r(y)} \cdot
    \delta_{s(x')=r(y')} \cdot
    \delta_{x\cdot x'=y\cdot y'} \cdot \delta_{s(x\cdot y)=v}.
  \end{multline*}
  On the other hand, we compute
  \begin{multline*}
    \braket[\big]{\delta_x\otimes \delta_y}
    {\delta_{x'}\otimes \delta_{y'}}(v)
    = \braket[\big]{\delta_y}{\braket{\delta_x}{\delta_{x'}}\delta_{y'}}(v)
    = \delta_{x=x'} \braket{\delta_y}{\delta_{s(x)} \delta_{y'}}(v)
    \\= \delta_{x=x'}\delta_{s(x)=r(y')} \braket{\delta_y}{\delta_{y'}}(v)
    = \delta_{x=x'}\delta_{s(x)=r(y')}\delta_{y=y'} \delta_{s(y)=v}.
  \end{multline*}
  Since~\(\tilde\mu_{n,m}\)
  is injective, \(x\cdot x' = y\cdot y'\)
  if and only if \(x=x'\)
  and \(y=y'\).
  We assumed \(s(x\cdot y)= s(y)\)
  as well.  So both expressions above are equal.  This shows
  that~\(\mu_{n,m}\)
  preserves the inner products between the elements
  \(\delta_x \otimes \delta_y\)
  and \(\delta_{x'} \otimes \delta_{y'}\)
  for \((x,y),(x',y')\in E_n \times_{s,r} E_m\).
  Hence it induces an isometry
  \(\Hilm_n \otimes_A \Hilm_m \injto \Hilm_{n+m}\).

  The condition \(r_{n+m}(x\cdot y)=r_n(x)\)
  for all \((x,y)\in E_n \times_{s,r} E_m\)
  implies that~\(\mu_{n,m}\)
  is a left module homomorphism.  The condition
  \(s_{n+m}(x\cdot y)=s_m(y)\)
  implies that it is a right module homomorphism.
\end{proof}

The associativity condition in~\eqref{assoeqn} is clearly equivalent
to
\[
(x\cdot y) \cdot z = x\cdot (y \cdot z)
\]
for all \(x\in E_n\),
\(y\in E_m\),
\(z\in E_l\),
\(n,m,l\in\N_{>0}\).
Thus we obtain a weak partial product system when we add this to our
requirements.  Our choices for \(E_0\),
\(\tilde\mu_{0,m}\)
and~\(\tilde\mu_{m,0}\)
say that \(\Hilm_0=\Hilm[A]\)
and that \(\mu_{0,m}\)
and~\(\mu_{m,0}\)
are the canonical maps
\(\Hilm[A] \otimes_A \Hilm_m \injto \Hilm_m\)
and \(\Hilm_m \otimes_A \Hilm[A] \congto \Hilm_m\)
as in Lemma~\ref{ass}.

The disjoint union \(E\defeq \bigsqcup_{n\in\N} E_n\)
is a category with object set~\(V\),
using the multiplication maps~\(\tilde\mu_{n,m}\)
for \(n,m\in\N\).
The decomposition of~\(E\)
as a disjoint union may be encoded by the functor from~\(E\)
to the monoid~\((\N,+)\) which
maps elements of~\(E_n\)
to~\(n\).
We want to describe a representation~\((\omega_n)_{n\in\N}\)
of the weak partial product system above through representations of
this category.  The representation~\(\omega_n\)
of~\(\Hilm_n\)
is given by the operators \(T_x\defeq \omega_n(\delta_x)\)
for \(x\in E_n\).
Since~\(\omega_0\)
is a representation of \(\Cont_0(V)\),
the operators~\(T_v\)
for \(v\in E_0=V\)
are orthogonal projections.
Conditions \ref{def:weak_representation_1}
and~\ref{def:weak_representation_2} in
Definition~\ref{def:weak_representation} say, first, that~\(T_x T_y\)
is \(T_{x\cdot y}\)
if \(s(x)=r(y)\),
and~\(0\)
otherwise and, secondly, that \(T_x^* T_y =\delta_{x,y} T_{s(y)}\)
if \(x,y\in E_n\)
for the same \(n\in\N\).
So each~\(T_x\)
is an isometry from \(T_{s(x)} B\subseteq B\)
into \(T_{r(x)} B\subseteq B\),
and the ranges of these isometries for \(x\in E_n\)
with fixed \(n\in\N\)
are orthogonal.  If~\(V\)
has only one element, then our category~\(E\)
becomes a monoid, and the map \(x\mapsto T_x\)
is a representation of this monoid by isometries with the extra
property that the isometries~\(T_x\)
for \(x\in E_n\) with fixed \(n\in\N\) have orthogonal ranges.

We now examine the Fock representation.  The Fock module~\(\Hilm[F]\)
is the Hilbert \(\Cont_0(V)\)-module
with basis \((\delta_x)_{x\in E}\)
with \(E=\bigsqcup E_n\)
as above, and with the inner product
\(\braket{\delta_x}{\delta_y} = \delta_{x=y} \delta_{s(x)}\).
The creation operator for \(x\in E\)
acts on~\(\Hilm[F]\)
by \(S_{\delta_x}(\delta_y) \defeq \delta_{x\cdot y}\)
for \(x,y\in E\).
Since the multiplication map on \(E_n \times_{s,r} E_m\)
is injective for each \(n,m\in\N\)
by assumption, the map~\(S_{\delta_x}\)
is a partial isometry with domain spanned by~\(\delta_y\)
with \(r(y)=s(x)\)
and image spanned by \(\delta_{x\cdot y}\)
for all such~\(y\).
This image is complementable, the complement being spanned by
those~\(\delta_y\)
with \(y\in E \setminus (x\cdot E)\).
So~\(S_{\delta_x}\) has the adjoint
\[
S_{\delta_x}^* (\delta_y) \defeq
\begin{cases}
  z&\text{if }y=x\cdot z\text{ for some }z\in E,\\
  0&\text{otherwise.}
\end{cases}
\]
Since the adjointable operators form a Banach space, it follows
that~\(S_\xi\)
is adjointable for all \(\xi\in \Hilm_n\).
Hence the Fock representation exists as a weak representation.

\section{Representations and partial product systems}
\label{sec:representation}

In this section, we restrict attention to weak partial product
systems for which the weak Fock representation exists.  We show that
the extra conditions for a representation,
\ref{def:weak_representation_4} and~\ref{def:weak_representation_5}
in Definition~\ref{def:weak_representation}, are equivalent
to~\eqref{eq:representation_condition}.  We show an example as in
Section~\ref{sec:graph} for which these conditions fail for the Fock
representation.  We define \emph{partial product systems} by
requiring that the Fock representation be defined and be a
representation.  We study the extra conditions needed for
representations for global product systems and Fell bundles.
Finally, we relate our notion of representation and partial product
system to Nica covariance and compact alignment for product systems
over quasi-lattice orders.

\begin{proposition}
  \label{simplifiability}
  Let \(\Hilm=(A,\Hilm_n,\mu_{n,m})_{n,m\in\N}\)
  be a weak partial product system for which the weak Fock
  representation exists, that is, the creation operators on its Fock
  module are adjointable.  Let \((\omega_n)_{n\in\N}\)
  be a weak representation of~\(\Hilm\)
  in a \(\Cst\)\nb-algebra~\(B\).
  Define \(\Theta_{m,n}\colon \Comp(\Hilm_m,\Hilm_n) \to B\)
  as in~\eqref{eq:def_Theta}.  The following are equivalent:
  \begin{enumerate}
  \item \label{simplifiability_1}%
    \(\omega_n(x)^* \omega_m(y) = \omega_{m-n}(S_n(x)^* y)\)
    for all \(n,m\in\N_{\ge1}\)
    with \(m>n\) and \(x\in \Hilm_n\), \(y\in \Hilm_m\);
  \item \label{simplifiability_2}%
    \(\omega_n(x)^* \omega_m(y) = \omega_{n-m}(S_m(y)^* x)^*\)
    for all \(n,m\in\N_{\ge1}\)
    with \(n>m\) and \(x\in \Hilm_n\), \(y\in \Hilm_m\);
  \item \label{simplifiability_3}%
    \(\omega_n(\Hilm_n)^* \cdot \omega_m(\Hilm_m) \subseteq
    \omega_{m-n}(\Hilm_{m-n})\)
    for all \(n,m\in\N_{>0}\)
    with \(m>n\);
  \item \label{simplifiability_4}%
    \(\omega_n(\Hilm_n)^* \cdot \omega_m(\Hilm_m) \subseteq
    \omega_{n-m}(\Hilm_{n-m})^*\)
    for all \(n,m\in\N_{>0}\)
    with \(m<n\);
  \item \label{simplifiability_5}%
    \(\omega_n(\Hilm_n)^* \cdot \omega_m(\Hilm_m) \cdot B \subseteq
    \omega_{m-n}(\Hilm_{m-n})\cdot B\)
    for all \(n,m\in\N\) with \(m>n\);
  \item \label{simplifiability_6}%
    \(\omega_n(\Hilm_n)^* \cdot \omega_m(\Hilm_m) \cdot B \subseteq
    \omega_{n-m}(\Hilm_{n-m})^*\cdot B\)
    for all \(n,m\in\N_{>0}\)
    with \(m<n\).
  \item \label{simplifiability_7}%
    If \(m,n,p,q\in\N\), then
    \[
    \Theta_{m,n}(\Comp(\Hilm_m,\Hilm_n)) \cdot
    \Theta_{p,q}(\Comp(\Hilm_p,\Hilm_q)) \subseteq
    \begin{cases}
      \Theta_{m-q+p,n}(\Comp(\Hilm_{m-q+p},\Hilm_n))&\text{if }m\ge q,\\
      \Theta_{p,n+q-m}(\Comp(\Hilm_p,\Hilm_{n+q-m}))&\text{if }m\le q.\\
    \end{cases}
    \]
  \end{enumerate}
  These equivalent conditions characterise
  when~\((\omega_n)_{n\in\N}\) is a representation.
\end{proposition}

\begin{proof}
  The condition in \ref{simplifiability_2} is the adjoint of the
  condition in~\ref{simplifiability_1}.  So these two conditions are
  equivalent.  Similarly, \ref{simplifiability_3}
  and~\ref{simplifiability_4} are equivalent.  The implications
  \[
  {\ref{simplifiability_1}}
  \Longrightarrow
  {\ref{simplifiability_3}}
  \Longrightarrow
  {\ref{simplifiability_5}},\qquad
  {\ref{simplifiability_2}}
  \Longrightarrow
  {\ref{simplifiability_4}}
  \Longrightarrow
  {\ref{simplifiability_6}}
  \]
  are trivial.  We are going to show that~\ref{simplifiability_5}
  implies~\ref{simplifiability_1}.  An analogous argument shows
  that~\ref{simplifiability_6} implies~\ref{simplifiability_2}.  This
  will complete the proof that conditions
  \ref{simplifiability_1}--\ref{simplifiability_6} are equivalent.
  A representation is a weak representation that also satisfies the
  conditions \ref{simplifiability_5} and~\ref{simplifiability_6}.
  Hence each of our equivalent conditions characterises representations
  among weak representations.

  Let \(z\in \Hilm_{m-n}\).  Then
  \begin{multline*}
    \omega_{m-n}(z)^*\omega_{m-n}(S_n(x)^*(y))
    = \omega_0(\braket{z}{S_n(x)^*(y)})
    = \omega_0(\braket{S_n(x)(z)}{y})
    \\ = \omega_0(\braket{x\cdot z}{y})
    = \omega_m(x\cdot z)^*\omega_m(y)
    = \omega_{m-n}(z)^*\omega_n(x)^*\omega_m(y).
  \end{multline*}
  Letting
  \[
  X\defeq \omega_{m-n}(S_n(x)^*(y)) - \omega_n(x)^*\omega_m(y),
  \]
  this becomes \(\omega_{m-n}(z)^* \cdot X = 0\)
  for all \(z\in \Hilm_{m-n}\)
  or \(X^*\cdot \omega_{m-n}(\Hilm_{m-n})\cdot B = 0\).
  Condition~\ref{simplifiability_5} implies
  \(X\cdot B\subseteq \omega_{m-n}(\Hilm_{m-n}) B\).
  Thus \(X^*\cdot X\cdot B=0\).  Hence \(X=0\).

  Condition~\ref{simplifiability_7} contains \ref{simplifiability_3}
  and~\ref{simplifiability_4} as special cases because
  \[
  \omega_n(\Hilm_n) = \Theta_{0,n}(\Comp(\Hilm_0,\Hilm_n)),\qquad
  \omega_n(\Hilm_n)^* = \Theta_{n,0}(\Comp(\Hilm_n,\Hilm_0)).
  \]
  It remains to prove that \ref{simplifiability_3} implies
  \ref{simplifiability_7}.  By definition,
  \[
  \Theta_{m,n}(\Comp(\Hilm_m,\Hilm_n)) =
  \omega_n(\Hilm_n) \omega_m(\Hilm_m)^*.
  \]
  Hence
  \begin{align*}
    \Theta_{m,n}(\Comp(\Hilm_m,\Hilm_n)) \cdot
    \Theta_{p,q}(\Comp(\Hilm_p,\Hilm_q))
    = \omega_n(\Hilm_n) \cdot \bigl(\omega_m(\Hilm_m)^*
    \omega_q(\Hilm_q)\bigr) \cdot \omega_p(\Hilm_p)^*.
  \end{align*}
  Now we apply~\ref{simplifiability_3} to the two
  factors in the middle.  If \(m>q\), then
  \begin{multline*}
    \omega_n(\Hilm_n) \omega_m(\Hilm_m)^* \omega_q(\Hilm_q)
    \omega_p(\Hilm_p)^*
    \subseteq \omega_n(\Hilm_n) \omega_{m-q}(\Hilm_{m-q})^*
    \omega_p(\Hilm_p)^*
    \\\subseteq \omega_n(\Hilm_n) \omega_{m-q+p}(\Hilm_{m-q+p})^*
    = \Theta_{m-q+p,n}(\Comp(\Hilm_{m-q+p},\Hilm_n)),
  \end{multline*}
  both for \(q>0\) and \(q=0\).  If \(m<q\), then
  \begin{multline*}
    \omega_n(\Hilm_n) \omega_m(\Hilm_m)^* \omega_q(\Hilm_q)
    \omega_p(\Hilm_p)^*
    \subseteq \omega_n(\Hilm_n) \omega_{q-m}(\Hilm_{q-m})
    \omega_p(\Hilm_p)^*
    \\\subseteq \omega_{n+q-m}(\Hilm_{n+q-m}) \omega_p(\Hilm_p)^*
    = \Theta_{p,n+q-m}(\Comp(\Hilm_p,\Hilm_{n+q-m})),
  \end{multline*}
  both for \(m>0\)
  and \(m=0\).
  The case \(m=q\)
  also works.  Thus~\ref{simplifiability_3}
  implies~\ref{simplifiability_7}.
\end{proof}

Equation~\eqref{eq:Theta_multiplicative} and
condition~\ref{simplifiability_7} in
Proposition~\ref{simplifiability} imply that
\[
\sum_{m,n=0}^\infty \Theta_{m,n}(\Comp(\Hilm_m,\Hilm_n))\subseteq B
\]
is a \Star{}subalgebra for any representation.  For \(n=m\)
and \(p=q\),
the condition~\ref{simplifiability_7} in
Proposition~\ref{simplifiability} says that
\begin{equation}
  \label{eq:multiply_Theta_nn}
  \Theta_n(\Comp(\Hilm_n)) \cdot \Theta_p(\Comp(\Hilm_p)) \subseteq
  \Theta_{\max \{n,p\}}(\Comp(\Hilm_{\max \{n,p\}}));
\end{equation}
here we abbreviated \(\Theta_n \defeq \Theta_{n,m}\).
So \(\sum_{n=0}^N \Theta_n(\Comp(\Hilm_n))\)
is a \Star{}subalgebra as well for all \(N\in\N\cup\{\infty\}\).

The conditions \ref{simplifiability_5} and~\ref{simplifiability_6}
concern inclusions between certain right ideals (or Hilbert
submodules) of~\(B\).
Hence they are similar to nondegeneracy conditions.  Thus
Proposition~\ref{simplifiability} is similar in spirit to
\cite{Albandik-Meyer:Product}*{Proposition~2.5}.

A \emph{partial product system} is a weak partial product system for
which the weak Fock representation exists and is a representation,
that is, satisfies the equivalent conditions in
Proposition~\ref{simplifiability}.  These assumptions will be used in
the next section to define and study the Toeplitz algebra of a partial
product system.

We now examine the difference between weak representations and
representations for several classes of weak partial product systems.
First, we examine product systems, then restrictions of Fell bundles
over~\(\Z\).
Finally, we study examples coming from graphs as in
Section~\ref{sec:graph}.

A \emph{product system} is a weak partial product system
\((A,\Hilm_n,\mu_{n,m})_{n,m\in\N}\)
where the multiplication maps
\(\mu_{n,m}\colon \Hilm_n \otimes_A \Hilm_m \injto \Hilm_{n+m}\)
for \(n,m\in\N_{>0}\) are unitary.

\begin{proposition}
  \label{pro:simplifiable_global}
  Any weak representation of a product system is a representation.
  Product systems are partial product systems.
\end{proposition}

\begin{proof}
  Let \((A,\Hilm_n,\mu_{n,m})_{n,m\in\N}\)
  be a product system and let \((\omega_n)_{n\in\N}\)
  be a weak representation of it in a \(\Cst\)\nb-algebra~\(B\).
  Let \(n,k\in\N_{>0}\).
  Then
  \(\omega_n(\Hilm_n)\cdot \omega_k(\Hilm_k) =
  \omega_{n+k}(\Hilm_{n+k})\),
  that is, the closed linear span of
  \(\omega_n(\xi)\cdot \omega_k(\eta)\)
  for \(\xi\in\Hilm_n\),
  \(\eta\in\Hilm_k\)
  is dense in \(\omega_{n+k}(\Hilm_{n+k})\).
  If \(m>n\), write \(m=n+k\).  Then
  \[
  \omega_n(\Hilm_n)^* \cdot \omega_m(\Hilm_m)
  = \omega_n(\Hilm_n)^* \cdot \omega_n(\Hilm_n)\cdot \omega_k(\Hilm_k)
  \subseteq \omega_0(A)\cdot \omega_k(\Hilm_k)
  \subseteq \omega_{m-n}(\Hilm_{m-n}).
  \]
  So any weak representation satisfies the
  condition~\ref{simplifiability_1} in
  Proposition~\ref{simplifiability}.  The Fock representation exists
  as a weak representation because the maps~\(\mu_{n,m}\)
  for \(n>0\)
  are adjointable.  It is a representation by the statement already
  proved.  That is, our product system is a partial product system.
\end{proof}

A \emph{Fell bundle} over the group~\(\Z\)
is given by Banach spaces~\(B_n\)
for \(n\in\Z\)
with multiplication maps \(B_n \times B_m \to B_{n+m}\)
and involutions \(B_n \to B_{-n}\)
with certain properties (see
\cite{Exel:Partial_dynamical}*{Definition~16.1}).  These properties are
equivalent to the existence of injective maps \(B_n \injto B\)
for some \(\Cst\)\nb-algebra~\(B\)
such that the multiplication maps and involutions
in~\((B_n)_{n\in\N}\)
are the restrictions of the multiplication and involution in~\(B\).
In particular, \(B_0\)
is a \(\Cst\)\nb-algebra.
Each~\(B_n\)
becomes a Hilbert \(B_0\)\nb-bimodule
by \(\BRAKET{b}{c} \defeq b c^*\)
and \(\braket{b}{c} \defeq b^* c\)
for all \(b,c\in B_n\).
The Fell bundle multiplication maps \(B_n \times B_m \to B_{n+m}\)
induce isometries
\(\mu_{n,m}\colon B_n \otimes_{B_0} B_m \injto B_{n+m}\)
of \(B_0,B_0\)\nb-correspondences
for all \(n,m\in\N\).

\begin{proposition}
  \label{pro:pps_from_Fell}
  Let \(\mathcal{B}= (B_n)_{n\in\Z}\)
  be a Fell bundle.  Its restriction to~\(\N\)
  is a partial product system; here each~\(B_n\)
  is viewed as a \(B_0,B_0\)-correspondence,
  even a Hilbert \(B_0\)\nb-bimodule,
  as above.
\end{proposition}

\begin{proof}
  The creation operator
  \(S_n(x)\colon B_m \to B_{n+m}\),
  \(y\mapsto x\cdot y\),
  for \(x\in B_n\)
  is adjointable with adjoint \(S_n(x)^*(y) = x^*\cdot y \).  So
  \[
  S_n(x)^* S_m(y) z
  = x^*\cdot (y\cdot z)
  = (x^*\cdot y)\cdot z
  = S_{m-n}(S_n(x)^*y)(z)
  \]
  for \(n,m,k\in\N\)
  with \(n\le m\)
  and \(x\in B_n\),
  \(y\in B_m\),
  \(z\in B_k\).
  Thus the Fock representation~\((S_n)_{n\in\N}\)
  satisfies condition~\ref{simplifiability_1} in
  Proposition~\ref{simplifiability}.  So it is a representation and we
  have got a partial product system over~\(\N\).
\end{proof}

\begin{definition}
  \label{fbrep}
  Let~\(D\) be a \(\Cst\)\nb-algebra.
  A \emph{representation} of a Fell bundle
  \(\mathcal{B}= (B_n)_{n\in\Z}\)
  in~\(D\)
  consists of linear maps \(\omega_n\colon B_n \to D\)
  for \(n\in\Z\) such that
  \begin{enumerate}
  \item \label{fbrep1}%
    \(\omega_0\colon B_1\to D\)
    is a \Star{}homomorphism;
  \item \label{fbrep2}%
    \(\omega_n(b) \omega_m(c) = \omega_{n+m} (bc)\)
    for all \(n,m\in\Z\), \(b\in B_n\), \(c\in B_m\);
  \item \(\omega_{-n}(b^*) = \omega_n(b)^*\)
    for all \(n\in\Z\), \(b\in B_n\).
  \end{enumerate}
\end{definition}

It is immediate from the definitions that the restriction of a
representation of a Fell bundle over~\(\Z\)
to~\(\N\)
is a representation of the partial product system over~\(\N\).
The Fock representation, however, is never the restriction of a Fell
bundle representation.  So a Fell bundle over~\(\Z\)
has strictly fewer representations than its partial product system
restriction to~\(\N\).

When is a (weak) partial product system over~\(\N\)
the restriction of a Fell bundle over~\(\Z\)?
We shall answer this question in Theorem~\ref{the:extend_pps_to_Fell}.
In this section, we only study some easy aspects of this question.  A
necessary condition is that each~\(\Hilm_n\)
be a Hilbert \(A\)\nb-bimodule.
The following example show that this is not yet sufficient, even if we
also assume~\(\Hilm\) to be a partial product system.

\begin{example}
  \label{exa:Hilbi_bad_1}
  Let \(A=\C\oplus\C\)
  and let \(\varphi\)
  be the partial isomorphism on~\(A\)
  that maps the first summand identically onto the second summand.
  This partial isomorphism corresponds to the Hilbert
  \(A\)\nb-bimodule
  \(\Hilm[E]_\varphi=\C\)
  with the \(A\)\nb-bimodule
  structure
  \((a_1,a_2)\cdot x\cdot (b_1,b_2) \defeq a_2\cdot x\cdot b_1\)
  and the left and right inner products
  \(\BRAKET{x}{y} \defeq (0,x \overline{y})\),
  \(\braket{x}{y} \defeq (\overline{x} y,0)\)
  for \(a_1,a_2,x,y,b_1,b_2\in\C\).
  Since the range and source ideals of~\(\Hilm_\varphi\)
  are orthogonal, \(\Hilm_\varphi \otimes_A \Hilm_\varphi \cong 0\).
  Let \(\Hilm_n \defeq \Hilm_\varphi\)
  for all \(n\in\N_{\ge1}\)
  and let
  \(\mu_{n,m}\colon \Hilm_n \otimes_A \Hilm_m \cong 0 \injto
  \Hilm_{n+m}\)
  be the zero map for all \(n,m\in\N_{\ge1}\).
  We claim that this is a partial product system over~\(\N\).
  All its fibres are Hilbert \(A\)\nb-bimodules.
  The maps~\(\mu_{n,m}\)
  are isometric and associative.  The creation operator~\(S_n(x)\)
  for \(x\in\Hilm_n\)
  vanishes on~\(\Hilm_m\)
  for \(m\in\N_{\ge1}\)
  and maps \((b_1,b_2)\in \Hilm_0 = A\)
  to \(x\cdot b_1\in\Hilm_n\).
  Hence \(S_n(x)\)
  is adjointable, and its adjoint vanishes on~\(\Hilm_m\)
  for all \(m\neq n\).
  Thus the Fock representation satisfies
  condition~\ref{simplifiability_1} in
  Proposition~\ref{simplifiability}.  So we have a partial product
  system of Hilbert bimodules.  It cannot come from a Fell bundle
  over~\(\Z\),
  however.  A Fell bundle also has injective multiplication maps
  \[
  \Hilm_n^* \otimes_A \Hilm_m \cong \Hilm_{-n} \otimes_A \Hilm_m
  \injto \Hilm_{m-n}.
  \]
  We have
  \(\Hilm_n^* \otimes_A \Hilm_m = \Hilm_\varphi^* \otimes_A
  \Hilm_\varphi\),
  which is the Hilbert bimodule that corresponds to the identity map
  on the first summand in~\(A\).
  Since there is no non-zero map from this to~\(\Hilm_\varphi\),
  there is no multiplication map
  \(\Hilm_n^* \otimes_A \Hilm_m \injto \Hilm_{m-n}\) for \(m>n\).
\end{example}

\begin{remark}
  Any product system of Hilbert bimodules is the restriction of a Fell
  bundle over~\(\Z\).
  Indeed, a product system is determined by the correspondence
  \(\Hilm_1\),
  and so a product system of Hilbert bimodules is given by a single
  Hilbert bimodule.  The Fell bundle generated by it is described
  in~\cite{Abadie-Eilers-Exel:Morita_bimodules}.  This does not yet give
  all Fell bundles over~\(\Z\),
  however: we only get those Fell bundles that are semi-saturated, that
  is, have surjective multiplication maps~\(\mu_{n,m}\)
  for all \(n,m\ge0\).
\end{remark}

There is a weak partial product system of the type introduced in
Section~\ref{sec:graph} for which the weak Fock representation
exists but is not a representation:

\begin{example}
  \label{exa:graph_not_pp}
  Let
  \[
  E_0=V=\{v_0,v_1,v_2,v_3\},\quad
  E_1 = \{a\},\quad
  E_2 = \{c,d\},\quad
  E_3 = \{b\},\quad
  E_4 = \{e\},\quad
  \]
  and \(E_n=\emptyset\)
  for \(n\ge 5\)
  with \(a\colon v_1 \to v_3\),
  \(b\colon v_0\to v_1\),
  \(c\colon v_2\to v_3\),
  \(d\colon v_0 \to v_2\),
  and \(e=a\cdot b = c\cdot d\).
  This determines the range and source maps
  \(r_n,s_n\colon E_n\rightrightarrows V\)
  and the multiplication maps
  \(\tilde\mu_{n,m}\colon E_n \times_{s_n,r_m} E_m \injto E_{n+m}\)
  for all \(n,m\in\N\).
  The category \(\bigsqcup_{n\in\N} E_n\)
  is the one that describes commutative squares:
  \[
  \begin{tikzcd}
    v_0 \arrow[r, "b"] \arrow[d, "d"'] \arrow[dr, "e"] & v_1 \arrow[d, "a"] \\
    v_2 \arrow[r, "c"'] & v_3,
  \end{tikzcd}\qquad e = a\circ b = c\circ d.
  \]
  The resulting \(\Cst\)\nb-algebra
  is \(A=\C[V] = \C^4\).
  The Fock module over~\(A\) is the \(\C\)\nb-vector space with basis
  \[
  E \defeq \{v_0,v_1,v_2,v_3,a,b,c,d,e\}
  \]
  with the \(A\)\nb-bimodule
  structure given by the range and source maps and with
  \(\braket{x}{y} \defeq \delta_{x=y} s(x)\)
  for all \(x,y\in E\).
  The Fock representation of the weak partial product system is given
  on the basis vectors by the following table:
  \[
  \begin{array}{*{10}{c}}
    &v_0&v_1&v_2&v_3&a&b&c&d&e\\ \hline
    S_1(a)&0&a&0&0&0&e& 0&0&0\\
    S_2(c)&0&0&c&0&0&0& 0&e&0\\
    S_2(d)&d&0&0&0&0&0& 0&0&0\\
    S_3(b)&b&0&0&0&0&0& 0&0&0\\
    S_4(e)&e&0&0&0&0&0& 0&0&0\\
  \end{array}
  \]
  So \(S_1(\Hilm_1)\Hilm[F] =\C[a, e]\),
  \(S_2(\Hilm_2)\Hilm[F] = \C[c,d,e]\),
  and \(S_1(\Hilm_1)^*(S_2(\Hilm_2)\Hilm[F])=\C[b]\).
  This is not contained in \(S_1(\Hilm_1)\Hilm[F]\).
\end{example}

\begin{proposition}
  \label{pro:graph_pps}
  A weak partial product system of the form built in
  Section~\textup{\ref{sec:graph}} is a partial product system if and
  only if \(E=\bigsqcup_{n\in\N} E_n\)
  is the path category of some directed graph~\(\Gamma\),
  but where generators in~\(\Gamma\)
  may have degrees different from~\(1\).
\end{proposition}

\begin{proof}
  We have already seen that the weak Fock representation exists, that
  is, the operators~\(S_n(\xi)\)
  for \(\xi\in\Hilm_n\)
  are adjointable.  When is it a representation?  This means that
  \(S_n(x)^*S_m(y)z\in S_{m-n}(\Hilm_{m-n})\Hilm[F]\)
  for \(m,n\in\N\)
  with \(m>n\)
  and all \(x\in E_n\),
  \(y \in E_m\),
  \(z\in E_k\).
  By definition,
  \(S_n(\delta_x)^*S_m(\delta_y)(\delta_z) =
  S_n(\delta_x)^*(\delta_{y\cdot z})\)
  is equal to
  \(\braket{\delta_x}{\delta_s} \delta_t = \delta_{x=s} \delta_t\)
  if \(y \cdot z = s \cdot t\)
  for some \(s \in E_n\)
  and \(t\in E_{m+k-n}\),
  and~\(0\)
  otherwise.  So \(S_n(\delta_x)^*S_m(\delta_y)(\delta_z) = \delta_t\)
  if \(y \cdot z = x \cdot t\)
  and~\(0\)
  otherwise.  We need this to belong to
  \(S_{m-n}(\Hilm_{m-n})\Hilm[F]\),
  that is, to be either~\(0\)
  or of the form~\(\delta_{a\cdot b}\)
  for some \(a\in E_{m-n}\),
  \(b\in E_k\).
  Thus the weak Fock representation is a representation if and only if
  \(y\cdot z = x\cdot t\)
  for \(x\in E_n\),
  \(y \in E_m\),
  \(z\in E_k\),
  \(t\in E_{m+k-n}\)
  implies that \(t=a\cdot b\)
  for some \(a\in E_{m-n}\),
  \(b\in E_k\).
  Since we already assumed that the multiplication map
  \(E_n\times_{s,r} E_k \to E_{n+k}\)
  is injective, the equation \(y\cdot z = (x\cdot a)\cdot b\)
  implies \(z=b\)
  and \(y=x\cdot a\).
  That is, the equation \(y\cdot z = x\cdot t\)
  has only the trivial solutions with \(y=x\cdot a\)
  and \(t=a\cdot z\) in~\(E\).

  We call an element \(x\in\bigsqcup_{n=1}^\infty E\)
  \emph{irreducible} if the only product decompositions \(x=a\cdot b\)
  with \(a,b\in E\)
  are \(x=1_{r(x)}\cdot x\)
  and \(x=x\cdot 1_{s(x)}\).
  The irreducible elements in~\(E\)
  with the restriction of the range and source maps form a directed
  graph~\(\Gamma\).
  Since~\(E\)
  is a category, any path in the graph~\(\Gamma\)
  defines an element in~\(E\).
  This defines a functor from the path category of~\(\Gamma\)
  to~\(E\).
  This functor is the identity on objects.  It is surjective on arrows
  because if an arrow is not irreducible, we may write it as a
  non-trivial product of two strictly shorter arrows, and decomposing
  these as long as possible will eventually write our arrow as a
  product of irreducible arrows.

  In a path category, the equation \(y\cdot z = x\cdot t\)
  only has the trivial solutions.  Conversely, we claim that the
  functor from the path category of~\(\Gamma\)
  to~\(E\)
  is injective if the equation \(y\cdot z = x\cdot t\)
  has only the trivial solutions and the multiplication maps
  \(E_n \times_{s,r} E_m \to E_{n+m}\)
  are injective for all \(n,m\in\N\).
  Indeed, assume that two paths \(a_1 \dotsm a_k\)
  and \(b_1 \dotsm b_l\)
  in the graph~\(\Gamma\)
  are mapped to the same arrow in~\(E\).
  Since \(y\cdot z = x\cdot t\)
  has only trivial solutions and \(a_k\)
  and~\(b_l\)
  are irreducible, we must have \(a_k = b_l\)
  and \(a_1 \dotsm a_{k-1}=b_1 \dotsm b_{l-1}\).
  By induction, we conclude that our two paths are equal.  So the
  functor from the path category of~\(\Gamma\)
  to~\(E\)
  is an isomorphism of categories if~\(E\)
  gives a partial product system.  The path category of~\(\Gamma\)
  differ from~\(E\)
  only through the grading, that is, the functor to~\(\N\).
  Whereas all generators in a usual path category have length~\(1\),
  they may belong to~\(E_n\) for any \(n\in\N\).
\end{proof}

Roughly speaking, Proposition~\ref{pro:graph_pps} says that the ``partial''
analogues of graph \(\Cst\)\nb-algebras
are not more general than graph \(\Cst\)\nb-algebras.
The only thing that is modified is the gauge action because the edges
of the graph~\(\Gamma\)
may belong to~\(\Hilm_n\)
and thus have degree~\(n\) for any \(n\in\N\).

\subsection{Analogy with Nica covariance}
\label{sec:Nica}

We briefly discuss the analogy between the extra condition for a weak
representation to be a representation and Nica covariance for
representations of product systems over quasi-lattice orders (see
\cite{Fowler:Product_systems}*{Definition~5.1}).  Let \((G,P)\)
be a quasi-lattice ordered group; for \(p,q\in P\),
let \(p\vee q\in P\cup\{\infty\}\)
be their least upper bound or~\(\infty\)
if \(p,q\)
have no upper bound in~\(P\).
Let \((A,\Hilm_p,\mu_{p,q})_{p,q\in P}\)
be a product system over~\(P\).
Let \((\omega_p)_{p\in P}\)
be a Toeplitz representation of the product system in a
\(\Cst\)\nb-algebra~\(B\).
The representation is called \emph{Nica covariant} if, for all
\(p,q\in P\),
\begin{equation}
  \label{eq:Nica_1}
  \omega_p(\Hilm_p) \cdot \omega_p(\Hilm_p)^* \cdot
  \omega_q(\Hilm_q) \cdot \omega_q(\Hilm_q)^*
  \subseteq
  \begin{cases}
    \omega_{p\vee q}(\Hilm_{p\vee q}) \cdot
    \omega_{p\vee q}(\Hilm_{p\vee q})^*&p\vee q\in P,\\
    0&p\vee q=\infty.
  \end{cases}
\end{equation}
Here we use
\(\Theta_p(\Comp(\Hilm_p)) = \omega_p(\Hilm_p)\omega_p(\Hilm_p)^*\)
by~\eqref{eq:def_Theta}.

\begin{lemma}
  \label{lem:Nica_covariant}
  A Toeplitz representation \((\omega_p)_{p\in P}\)
  of a product system over \((G,P)\)
  satisfies~\eqref{eq:Nica_1} if and only if
  \begin{equation}
    \label{eq:Nica_2}
    \omega_p(\Hilm_p)^* \omega_q(\Hilm_q)
    \subseteq
    \begin{cases}
      \omega_{p^{-1}(p\vee q)}(\Hilm_{p^{-1}(p\vee q)})
      \omega_{q^{-1}(p\vee q)}(\Hilm_{q^{-1}(p\vee q)})^*&
      p\vee q\in P,\\
      0&
      p\vee q=\infty.
    \end{cases}
  \end{equation}
\end{lemma}

\begin{proof}
  Equation~\eqref{eq:Nica_2} multiplied on the left by
  \(\omega_p(\Hilm_p)\)
  and on the right by \(\omega_q(\Hilm_q)^*\)
  becomes~\eqref{eq:Nica_1} because
  \(\omega_p(\Hilm_p) \cdot \omega_{p^{-1}(p\vee
    q)}(\Hilm_{p^{-1}(p\vee q)}) = \omega_{p\vee q}(\Hilm_{p\vee q})\)
  as in the proof of Proposition~\ref{pro:simplifiable_global}.
  Conversely, multiply~\eqref{eq:Nica_1} on the left by
  \(\omega_p(\Hilm_p)^*\)
  and on the right by \(\omega_q(\Hilm_q)\).  We may simplify
  \begin{multline*}
    \omega_p(\Hilm_p)^* \cdot \omega_p(\Hilm_p) \cdot \omega_p(\Hilm_p)^*
    = \omega_0(\braket{\Hilm_p}{\Hilm_p}) \cdot \omega_p(\Hilm_p)^*
    \\= \omega_p(\Hilm_p \cdot \braket{\Hilm_p}{\Hilm_p})^*
    = \omega_p(\Hilm_p)^*
  \end{multline*}
  and similarly for~\(q\).  And
  \begin{multline*}
    \omega_p(\Hilm_p)^* \cdot \omega_{p\vee q}(\Hilm_{p\vee q})
    = \omega_p(\Hilm_p)^* \cdot \omega_p(\Hilm_p) \cdot
    \omega_{p^{-1}(p\vee q)}(\Hilm_{p^{-1}(p\vee q)})
    \\= \omega_0(\braket{\Hilm_p}{\Hilm_p}) \cdot \omega_{p^{-1}(p\vee q)}(\Hilm_{p^{-1}(p\vee q)})
    = \omega_{p^{-1}(p\vee q)}(\braket{\Hilm_p}{\Hilm_p} \cdot \Hilm_{p^{-1}(p\vee q)})
    \\\subseteq \omega_{p^{-1}(p\vee q)}(\Hilm_{p^{-1}(p\vee q)}).
  \end{multline*}
  In this way, \eqref{eq:Nica_1} implies~\eqref{eq:Nica_2}.
\end{proof}

Now let \((G,P)\)
be \((\Z,\N)\).
If \(p,q\in\N\),
then \(p^{-1}(p\vee q) = \max\{p,q\} - p\)
is \(0\)
if \(p\ge q\)
and \(q-p\)
if \(p\le q\).
So~\eqref{eq:Nica_2} becomes
\(\omega_p(\Hilm_p)^* \cdot \omega_q(\Hilm_q) \subseteq
\omega_{q-p}(\Hilm_{q-p})\)
if \(p\le q\)
and
\(\omega_p(\Hilm_p)^* \cdot \omega_q(\Hilm_q) \subseteq
\omega_{p-q}(\Hilm_{p-q})^*\)
if \(p\ge q\).
Proposition~\ref{simplifiability} shows that each of these conditions
is equivalent to the conditions in
Definition~\ref{def:weak_representation}.
Proposition~\ref{pro:simplifiable_global} shows that~\eqref{eq:Nica_2}
holds automatically for weak representations of a global product
system.  In other words, any Toeplitz representation of a product
system over~\(\N\)
is Nica covariant, which is
\cite{Fowler:Product_systems}*{Remark~5.2}.

\section{The Toeplitz algebra}
\label{sec:Toeplitz}

In this section, we define the Toeplitz \(\Cst\)\nb-algebra
as the universal \(\Cst\)\nb-algebra
for representations of a partial product system.  We equip the
Toeplitz \(\Cst\)\nb-algebra
with a gauge action of the circle group~\(\T\),
describe the spectral subspaces, and prove a gauge-equivariant
uniqueness theorem.  We describe the fixed-point subalgebra of the
gauge action as an inductive limit and then show that the Fock
representation generates a faithful representation of the Toeplitz
\(\Cst\)\nb-algebra.

\begin{definition}
  Let \(\Hilm = (A, (\Hilm_n)_n, \mu_{n,m})\)
  be a partial product system.  Its Toeplitz algebra is a
  \(\Cst\)\nb-algebra~\(\Toep\)
  with a representation \((\bar{\omega}_n)_{n\in\N}\)
  that is universal in the following sense: for any representation
  \((\omega_n)_{n\in\N}\)
  of~\(\Hilm\)
  in a \(\Cst\)\nb-algebra~\(B\),
  there is a unique
  \Star{}homomorphism \(\varrho \colon \Toep \to B\)
  with \(\varrho \circ \bar{\omega}_n = \omega_n\) for all \(n\in\N\).
\end{definition}

For a product system in the usual sense, any weak representation is a
representation by Proposition~\ref{pro:simplifiable_global}.  Hence
the universal property of the Toeplitz \(\Cst\)\nb-algebra
is the usual one in this case.  So our definition of the Toeplitz
\(\Cst\)\nb-algebra
generalises the usual definition for product systems.

\begin{proposition}
  \label{pro:Toeplitz_relations}
  Any partial product system \((A, \Hilm_n, \mu_{n,m})\)
  has a Toeplitz algebra.  It is the universal \(\Cst\)\nb-algebra
  generated by the symbols \(\bar{\omega}_n(x)\)
  for \(n\in \N\),
  \(x \in \Hilm_n\), subject to the following relations:
  \begin{enumerate}
  \item \label{Toeplitz_relation_1}%
    \(\bar{\omega}_0\) is a \Star{}homomorphism;
  \item \label{Toeplitz_relation_2}%
    the maps \(x\mapsto \bar{\omega}_n(x)\)
    are linear for all \(n\in\N_{\ge1}\);
  \item \label{Toeplitz_relation_3}%
    \(\bar{\omega}_n(x) \bar{\omega}_m(y) =
    \bar{\omega}_{n+m}(\mu_{n,m}(x\otimes y))\)
    for \(x \in \Hilm_n\), \(y \in \Hilm_m\), \(n,m\in \N\);
  \item \label{Toeplitz_relation_4}%
    \(\bar{\omega}_n(x)^* \bar{\omega}_n(y) =
    \bar{\omega}_0(\braket{x}{y}_A)\)
    for all \(x \in \Hilm_n\), \(y \in \Hilm_n\) \(n \in \N\);
  \item \label{Toeplitz_relation_5}%
    \(\omega_n(x)^*\omega_m(y)=\omega_{m-n}(S_n(x)^*(y))\) for
    all \(n,m\in\N\) with \(m>n\) and \(x\in\Hilm_n\), \(y\in\Hilm_m\).
  \end{enumerate}
\end{proposition}

\begin{proof}
  Let~\(\mathfrak{T}\)
  be the \Star{}algebra with the generators and relations as above.
  Any \(\Cst\)\nb-norm on~\(\mathfrak{T}\) satisfies
  \[
  \norm{\bar{\omega}_n(x)}
  = \norm{\bar{\omega}_n(x)^*\bar{\omega}_n(x)}^{1/2}
  = \norm{\bar{\omega}_0(\braket{x}{x}_A)}^{1/2}
  \le \norm{\braket{x}{x}_A}^{1/2}
  = \norm{x}_{\Hilm_n},
  \]
  where the inequality uses that~\(\bar{\omega}_0\)
  is a \Star{}homomorphism.  Hence the set of all \(\Cst\)\nb-norms
  on~\(\mathfrak{T}\)
  has a maximum.  We claim that the completion of~\(\mathfrak{T}\)
  in this maximal \(\Cst\)\nb-norm
  is the Toeplitz algebra of the partial product system.
  The relations that define~\(\mathfrak{T}\)
  are exactly those that are needed to make~\(\bar{\omega}_n\)
  a representation of our partial product system in~\(\Toep\).
  Here we use Proposition~\ref{simplifiability}, which shows that
  \ref{Toeplitz_relation_5} implies
  \(\omega_n(x)^*\omega_m(y)=\omega_{n-m}(S_m(y)^*(x))^*\)
  for \(m<n\),
  \(x\in\Hilm_n\),
  \(y\in\Hilm_m\).
  Any representation~\((\omega_n)_{n\in\N}\)
  of the partial product system in a \(\Cst\)\nb-algebra~\(B\)
  induces a unique \Star{}homomorphism \(\Toep \to B\)
  mapping \(\bar{\omega}_n(x) \mapsto \omega_n(x)\)
  for all \(n\in\N\), \(x\in\Hilm_n\).
\end{proof}

Next we define a gauge action on~\(\Toep\).
The maps
\[
\alpha_n\colon \Hilm_n \to \Cont(\T, \Toep),\qquad
x\mapsto (t \mapsto t^n \bar{\omega}_n(x)),
\]
form a representation of the partial product system
because~\(\bar{\omega}_n\)
is one.  By the universal property of~\(\Toep\),
there is a unique \Star{}homomorphism
\(\alpha\colon \Toep \to \Cont(\T, \Toep)\)
with \(\alpha(\bar{\omega}_n(x)) = \alpha_n(x)\)
for all \(n\in\N\),
\(x\in \Hilm_n\).
Define \(\alpha_t\colon \Toep \to \Toep\)
for \(t\in \T\)
by \(\alpha_t\defeq \ev_t\circ \alpha\).
Since~\(\alpha\)
is a \Star{}homomorphism, each~\(\alpha_t\)
is a \Star{}homomorphism and \(t\mapsto \alpha_t(x)\)
is continuous for each \(x\in\Toep\).
Since
\(\alpha_t(\alpha_s(\bar{\omega}_n(x))) = (ts)^n \bar{\omega}_n(x) =
\alpha_{ts}(\bar{\omega}_n(x))\)
and
\(\alpha_1(\bar{\omega}_n(x)) = \bar{\omega}_n(x) =
\id_\Toep(\bar{\omega}_n(x))\)
for \(x\in \Hilm_n\)
and \(n\in\N\),
the uniqueness part of the universal property of~\(\Toep\)
implies \(\alpha_t \alpha_s=\alpha_{t s}\)
for all \(t,s\in\T\)
and \(\alpha_1 = \id_\Toep\).
Thus each~\(\alpha_t\)
is bijective and \(\T \ni t \mapsto \alpha_t\)
is a continuous action of~\(\T\) on~\(\Toep\) by \Star{}automorphisms.

A circle action on a \(\Cst\)\nb-algebra
gives a lot of useful extra structure
(see~\cite{Exel:Circle_actions}).  We now use this for the
gauge action on~\(\Toep\).  We define its spectral subspaces
\[
\Toep_n \defeq \setgiven{x\in \Toep}
{\alpha_t(x) = t^n x \text{ for all } t\in \T}
\]
for \(n\in\Z\).
These spectral subspaces form a so-called \(\Z\)\nb-grading,
that is, \(\Toep_n \Toep_m \subseteq \Toep_{n+m}\)
and \(\Toep_n^* = \Toep_{-n}\)
for all \(n,m\in\Z\),
and the closed linear span of the subspaces~\(\Toep_n\)
is dense in~\(\Toep\).  In particular,
\[
\Toep_0 \defeq \setgiven{x \in \Toep}{\alpha_t(x) = x \text{ for all } t\in \T}
\]
is a \(\Cst\)\nb-subalgebra,
the \emph{fixed-point subalgebra} of~\(\alpha\).

We define the \emph{spectral projections}
\(E_n\colon \Toep \to \Toep\) for \(n\in \Z\) by
\[
E_n(x) \defeq \int_\T t^{-n} \alpha_t(x) \,\diff t
\]
for \(x\in \Toep\),
where~\(\diff t\)
denotes the normalised Haar measure on the compact group~\(\T\).
Each~\(E_n\)
is norm contractive and is an idempotent operator with
image~\(\Toep_n\)
that vanishes on~\(\Toep_m\)
for \(m\neq n\).
In particular, \(E_0\colon \Toep \to \Toep_0\)
is a conditional expectation.  It is well known to be faithful, that
is, if \(x\in\Toep\)
satisfies \(x\ge0\) and \(E_0(x)=0\), then \(x=0\).
This implies the following gauge-invariant uniqueness theorem:

\begin{theorem}
  \label{inje}
  Let~\(B\)
  be a \(\Cst\)\nb-algebra
  with an action~\(\beta\)
  of~\(\T\).
  Let \(\varrho\colon \Toep \to B\)
  be a \(\T\)\nb-equivariant
  \Star{}homomorphism.
  If~\(\varrho\)
  is injective on \(\Toep_0\subseteq \Toep\),
  then it is injective on~\(\Toep\).
\end{theorem}

\begin{proof}
  Let \(\varrho(x)=0\)
  for some \(x\in \Toep\).
  Let~\(E^B_0\)
  denote the \(0\)th
  spectral projection on~\(B\).
  Since~\(\varrho\)
  is \(\T\)\nb-equivariant,
  it intertwines the spectral projections \(E_0\)
  on~\(\Toep\) and~\(E_0^B\) on~\(B\).  So
  \[
  \varrho(E_0(x^*x))
  = E^B_0(\varrho (x^*x))
  = E^B_0(\varrho (x)^*\varrho(x))
  = 0.
  \]
  This implies \(E_0(x^*x)=0\)
  because \(E_0(x^*x)\in \Toep_0\),
  and \(\varrho|_{\Toep_0}\)
  is injective.  Then \(x=0\) because~\(E_0\) is faithful.
\end{proof}

Let \(\bar\Theta_{m,n}\colon \Comp(\Hilm_m,\Hilm_n) \to \Toep\)
be the maps defined by the representation~\(\bar{\omega}_n\)
as in~\eqref{eq:def_Theta}, that is,
\(\bar\Theta_{m,n}(\ket{x} \bra{y}) = \bar\omega_n(x) \bar\omega_m(y)^*\)
for \(n,m\in\N\),
\(x\in\Hilm_n\),
\(y\in\Hilm_m\).
The relations in Proposition~\ref{pro:Toeplitz_relations} imply that
any word in the generators of~\(\mathfrak{T}\)
may be reduced to one of the form
\(\bar{\omega}_n(x) \bar{\omega}_m(y)^*\)
for some \(n,m \in \N\),
\(x\in \Hilm_n\),
\(y\in \Hilm_m\).
Therefore, the closed linear span of
\(\bar\Theta_{m,n}(\Comp(\Hilm_m,\Hilm_n))\)
for \(n,m\in\N\) is a dense subspace in~\(\Toep\).
It is a \Star{}subalgebra by
Proposition~\ref{simplifiability}.\ref{simplifiability_7}.

\begin{lemma}
  \label{lem:Toeplitz_spectral}
  The spectral subspace \(\Toep_k\subseteq\Toep\)
  for \(k\in\Z\)
  is the closed linear span of
  \(\bar\Theta_{n,n+k}(\Comp(\Hilm_n,\Hilm_{n+k}))\)
  for all \(n\in\N\) with \(n\ge -k\).
\end{lemma}

\begin{proof}
  If \(x\in\Hilm_{n+k}\),
  \(y\in\Hilm_n\),
  then
  \(\alpha_t(\bar{\omega}_{n+k}(x) \bar{\omega}_n(y)^*) = t^k
  \bar{\omega}_{n+k}(x) \bar{\omega}_n(y)^*\).
  Thus \(\bar{\omega}_{n+k}(x) \bar{\omega}_n(y)^* \in \Toep_k\).
  Now~\eqref{eq:def_Theta} implies that
  \(\bar\Theta_{n,n+k}(\Comp(\Hilm_n,\Hilm_{n+k}))\)
  is contained in~\(\Toep_k\).
  We have already observed that the linear span of
  \(\bar\Theta_{n,m}(\Comp(\Hilm_n,\Hilm_m))\)
  is dense in~\(\Toep\).
  If \(x\in\Toep_k\)
  and \(\varepsilon>0\),
  then there is a finite linear combination
  \(\sum_{n,m\in\N} \bar\Theta_{n,m}(x_{n,m})\)
  with \(x_{n,m}\in \Comp(\Hilm_n,\Hilm_m)\)
  that is \(\varepsilon\)\nb-close
  to~\(x\).
  Then
  \[
  E_k \Biggl(\sum_{n,m\in\N} \bar\Theta_{n,m}(x_{n,m})\Biggr)
  = \sum_{n,n+k\in\N} \bar\Theta_{n,n+k}(x_{n,n+k})
  \]
  is still \(\varepsilon\)\nb-close
  to~\(x\).
  So the closed linear span of
  \(\bar\Theta_{n,n+k}(\Comp(\Hilm_n,\Hilm_{n+k}))\)
  is dense in~\(\Toep_k\).
\end{proof}

\begin{theorem}
  \label{the:Toeplitz_gauge_fixed}
  The map
  \[
  \bigoplus_{j=0}^N \bar\Theta_j\colon \bigoplus_{j=0}^N \Comp(\Hilm_j)
  \to \Toep_0
  \]
  is injective and its image~\(\Toep_{0,N}\)
  is a \(\Cst\)\nb-subalgebra
  of~\(\Toep_0\).
  These \(\Cst\)\nb-subalgebras
  for \(N\to\infty\)
  form an inductive system with colimit~\(\Toep_0\).
\end{theorem}

\begin{proof}
  The Fock representation~\((S_n)_{n\in\N}\)
  induces a \Star{}homomorphism
  \(\varphi\colon \Toep\to\Bound(\Hilm[F])\)
  with \(\varphi(\bar\omega_n(x)) = S_n(x)\)
  for all \(n\in\N\),
  \(x\in \Hilm_n\).
  Define the \Star{}homomorphisms
  \(\Theta_n\colon \Comp(\Hilm_n) \to \Bound(\Hilm[F])\)
  as in~\eqref{eq:def_Theta}.  Then
  \(\varphi(\bar\Theta_n(x)) = \Theta_n(x)\)
  for all \(n\in\N\),
  \(x\in \Comp(\Hilm_n)\).
  The operator~\(\Theta_n(x)\)
  is block diagonal on
  \(\Hilm[F]\defeq \bigoplus_{j=0}^\infty \Hilm_j\).
  Its first~\(n-1\)
  diagonal entries are~\(0\),
  and the \(n\)th
  diagonal entry is~\(x\)
  acting on~\(\Hilm_n\).
  Let \(x_j\in\Comp(\Hilm_j)\)
  for \(j=0,\dotsc,N\)
  satisfy \(\sum_{j=0}^N \Theta_j(x_j)=0\).
  We prove recursively that \(x_0=0\),
  \(x_1=0\),
  \(x_2=0\),
  and so on, by looking at the \(j\)th
  diagonal entry for \(j=0,\dotsc,N\).  So the composite map
  \[
  \bigoplus_{j=0}^N \Comp(\Hilm_j)
  \prto \Toep_{0,N}
  \subseteq \Toep
  \xrightarrow{\varphi} \Bound(\Hilm[F])
  \]
  is injective.  Even more, the recursive proof above shows that this
  injective map has a bounded inverse on its image.  So its image is
  closed in~\(\Toep_0\).
  Equation~\eqref{eq:multiply_Theta_nn} and
  Lemma~\ref{lem:Toeplitz_spectral} imply that~\(\Toep_{0,N}\)
  is a \Star{}subalgebra in~\(\Toep_0\).
  Since it is closed as well, it is a \(\Cst\)\nb-subalgebra.
  Lemma~\ref{lem:Toeplitz_spectral} shows that the union
  \(\bigcup_{N=0}^\infty \Toep_{0,N}\)
  is dense in~\(\Toep_0\).
  Hence~\(\Toep_0\)
  is the inductive limit \(\Cst\)\nb-algebra
  of the inductive system \((\Toep_{0,N})_{N\in\N}\).
\end{proof}

\begin{theorem}
  \label{the:Toeplitz_Fock}
  The Fock representation~\((S_n)_{n\in\N}\)
  on the Fock module~\(\Hilm[F]\)
  over~\(A\)
  induces a faithful representation of~\(\Toep\).
  So~\(\Toep\)
  is isomorphic to the \(\Cst\)\nb-subalgebra
  of \(\Bound(\Hilm[F])\)
  generated by \(S_n(\Hilm_n)\) for all \(n\in\N\).
\end{theorem}

\begin{proof}
  Let \(\varphi\colon \Toep\to\Bound(\Hilm[F])\)
  be the representation induced by the Fock representation; this
  exists because~\(\Hilm\) is a partial product system.
  The proof of Theorem~\ref{the:Toeplitz_gauge_fixed} shows that the
  restriction of~\(\varphi\)
  to \(\Toep_{0,N}\subseteq \Toep\)
  is injective for all \(N\in\N\).
  Since~\(\Toep_0\)
  is the inductive limit of these \(\Cst\)\nb-subalgebras by
  Theorem~\ref{the:Toeplitz_gauge_fixed}, it follows
  that~\(\varphi|_{\Toep_0}\)
  is injective.
  The Fock Hilbert module carries an obvious
  gauge action with spectral subspaces \(\Hilm[F]_n = \Hilm_n\).
  Let \(\beta\colon \T\to\Bound(\Hilm[F])\)
  be the induced action.  The Fock representation~\(\varphi\)
  is \(\T\)\nb-equivariant
  because \(S_n(x)\)
  belongs to the \(n\)th
  spectral subspace of~\(\Bound(\Hilm[F])\)
  for all \(n\in\N\),
  \(x\in\Hilm_n\).
  Hence Theorem~\ref{inje} shows that~\(\varphi\)
  is injective.  Its image is the \(\Cst\)\nb-subalgebra
  generated by the operators~\(S_n(x)\)
  for \(n\in\N\),
  \(x\in\Hilm_n\)
  because the elements \(\bar\omega_n(x)\)
  for \(n\in\N\),
  \(x\in\Hilm_n\)
  generate~\(\Toep\).
\end{proof}

\begin{proposition}
  \label{pro:sum_compacts_injects}
  The map \(\bigoplus_{m,n\in\N} \bar\Theta_{m,n}\colon \bigoplus_{m,n\in\N}
  \Comp(\Hilm_m,\Hilm_n) \to \Toep\)
  is injective.
\end{proposition}

\begin{proof}
  It suffices to prove that
  \(\bigoplus_{m,n} \Theta_{m,n}\colon \bigoplus_{m,n}
  \Comp(\Hilm_m,\Hilm_n) \to \Toep\injto \Bound(\Hilm[F])\)
  is injective.  We describe operators on~\(\Hilm[F]\)
  by block matrices.  If
  \(x\in\Comp(\Hilm_m,\Hilm_n)\),
  then the \(j,k\)-entry
  of \(\Theta_{m,n}(x)\)
  vanishes unless \(j-k=n-m\)
  and \(k\ge m\),
  and the \(n,m\)-entry
  is \(x\colon \Hilm_m \to \Hilm_n\).
  Let \(\sum_{m,n\in\N}\Theta_{m,n}(x_{m,n}) = 0\)
  for some \(x_{m,n}\in\Comp(\Hilm_m,\Hilm_n)\)
  with only finitely many non-zero~\(x_{n,m}\).
  Now we examine the \(j,k\)-entries of
  \(\sum_{m,n\in\N}\Theta_{m,n}(x_{m,n})\)
  for increasing~\(j\).
  For \(j=0\),
  we see that \(x_{m,0}=0\)
  for all \(m\in\N\).
  An induction over~\(j\)
  shows that
  \(x_{m,j}=0\)
  for all \(m\in\N\) and all \(j\in\N\).
\end{proof}

\subsection{Hereditary restrictions and quotients}
\label{sec:restrict_quotient}

Throughout this subsection, let~\(A\)
be a \(\Cst\)\nb-algebra
and let \(\Hilm\defeq (A,\Hilm_n,\mu_{n,m})_{n,m\in\N}\)
be a partial product system over~\(A\).
Let~\(\Toep\)
be its Toeplitz \(\Cst\)\nb-algebra.
We are going to restrict the partial product system to a hereditary
\(\Cst\)\nb-subalgebra
\(H\subseteq A\)
and a quotient~\(A/I\)
for an \emph{invariant} ideal~\(I\)
in~\(A\).
We show that the Toeplitz \(\Cst\)\nb-algebra
for the restriction to~\(A/I\) is a quotient of~\(\Toep\).

First let \(H\subseteq A\)
be a hereditary \(\Cst\)\nb-subalgebra.  Define
\[
\Hilm_n|_H \defeq H\cdot \Hilm_n\cdot H \subseteq \Hilm_n;
\]
the set of all products \(a\cdot x\cdot b\)
with \(a,b\in H\),
\(x\in\Hilm_n\)
is already a closed linear subspace of~\(\Hilm_n\)
by the Cohen--Hewitt Factorisation Theorem.  The multiplication maps
in the given partial product system~\((\Hilm_n)_{n\in\N}\)
restrict to multiplication maps
\[
\mu_{n,m}|_H\colon
\Hilm_n|_H \otimes_H \Hilm_m|_H \injto \Hilm_{n+m}|_H.
\]

\begin{lemma}
  \label{lem:restrict_hereditary}
  The Hilbert \(H,H\)\nb-modules
  \(\Hilm_n|_H\)
  and the multiplication maps~\(\mu_{n,m}|_H\)
  form a partial product system~\(\Hilm|_H\) over~\(H\).
\end{lemma}

\begin{proof}
  Let~\(\Hilm[F]\)
  be \(\bigoplus_{n=0}^\infty \Hilm_n\)
  with the Fock representation~\((S_n)_{n\in\N}\)
  of~\(\Hilm\).
  The Hilbert submodule
  \(\Hilm[F]\cdot H =\bigoplus_{n=0}^\infty \Hilm_n\cdot H\)
  is invariant for \(S_n(x)\)
  and \(S_n(x)^*\)
  for all \(x\in\Hilm_n\).
  The Hilbert submodule
  \(H\cdot \Hilm[F]\cdot H =\bigoplus_{n=0}^\infty H\cdot \Hilm_n\cdot H\)
  is still invariant for \(S_n(x)\)
  and \(S_n(x)^*\)
  for all \(x\in H\cdot \Hilm_n\cdot H\).
  We claim that \(S'_n(x) = S_n(x)|_{H\cdot \Hilm[F]\cdot H}\)
  for \(n\in\N\),
  \(x\in H\cdot \Hilm_n\cdot H\)
  defines a representation~\((S'_n)_{n\in\N}\)
  of~\((\Hilm_n|_H,\mu_{n,m}|_H)_{n,m\in\N}\)
  on~\(H\cdot \Hilm[F]\cdot H\).
  The operators~\(S'_n(x)\)
  for \(n\in\N\),
  \(x\in\Hilm_n\)
  are adjointable with adjoint
  \(S'_n(x)^* = S_n(x)^*|_{H\cdot \Hilm[F]\cdot H}\).
  The conditions in Definition~\ref{def:weak_representation} are
  inherited from the corresponding ones for the
  representation~\((S_n)_{n\in\N}\)
  of~\(\Hilm_n\).
  The representation~\(S'_n\)
  is unitarily equivalent to the Fock representation of
  \((\Hilm_n|_H,\mu_{n,m}|_H)_{n,m\in\N}\).
\end{proof}

Let \(I\idealin A\) be an invariant ideal.  Then
\[
\Hilm_n|_{A/I} \defeq \Hilm_n\bigm/\Hilm_n\cdot I
\]
is a Hilbert \(A/I,A/I\)-module
in a canonical way.  The left action of~\(A\)
descends to~\(A/I\)
because \(I\cdot \Hilm_n \subseteq \Hilm_n\cdot I\).
The multiplication map~\(\mu_{n,m}\)
induces a well defined multiplication map
\[
\mu_{n,m}|_{A/I}\colon
\Hilm_n|_{A/I} \otimes_{A/I} \Hilm_m|_{A/I}
\to \Hilm_{n+m}|_{A/I}.
\]

\begin{proposition}
  \label{pro:quotient_by_invariant_ideal}
  The restriction
  \(\Hilm|_{A/I} \defeq (A/I,\Hilm_n|_{A/I},\mu_{n,m}|_{A/I})_{n,m\in\N}\)
  is a partial product system.  Its Toeplitz \(\Cst\)\nb-algebra
  is a quotient of~\(\Toep\) by a \(\T\)\nb-invariant ideal.
\end{proposition}

\begin{proof}
  The beginning of the following proof works for any ideal
  \(I\idealin A\)
  and will later be used in this generality.  Let~\(\Hilm[F]\)
  be \(\bigoplus_{n=0}^\infty \Hilm_n\)
  with the Fock representation~\((S_n)_{n\in\N}\)
  of~\(\Hilm\).
  The Hilbert submodule
  \(\Hilm[F]\cdot I = \bigoplus_{n\in\N} \Hilm_n\cdot I\)
  is invariant for all adjointable operators on~\(\Hilm[F]\).
  Hence there is a unital \Star{}homomorphism
  \[
  \pi\colon \Bound(\Hilm[F]) \to \Bound(\Hilm[F]/\Hilm[F] I)
  \]
  We may identify
  \(\Hilm[F]/\Hilm[F] I \cong \Hilm[F] \otimes_A A/I\),
  and then \(\pi(T) = T\otimes 1\)
  for all \(T\in\Bound(\Hilm[F])\).
  The Hilbert modules \(\Hilm[F]\)
  and~\(\Hilm[F]/\Hilm[F]I\)
  carry obvious \(\Z\)\nb-gradings,
  which induce actions of the circle group~\(\T\)
  on \(\Bound(\Hilm[F])\)
  and \(\Bound(\Hilm[F]/\Hilm[F] I)\)
  that are continuous in the strict topology.  The
  homomorphism~\(\pi\)
  is grading-preserving.  Composing the Fock
  representation~\((S_n)_{n\in\N}\)
  with the \Star{}homomorphism~\(\pi\)
  gives a representation~\((S'_n)_{n\in\N}\)
  of~\(\Hilm\)
  on \(\Hilm[F]/\Hilm[F] I\).
  The \(\Cst\)\nb-algebra
  generated by~\(S'_n(\Hilm_n)\)
  is \(\pi(\Toep)\).
  Thus it is a quotient of~\(\Toep\).
  Since~\(\pi\)
  is \(\T\)\nb-equivariant,
  its kernel is a gauge-invariant ideal in~\(\Toep\).

  Now assume that~\(I\)
  is invariant.  Then \(S'_n|_{\Hilm_n I}=0\) because
  \[
  \Hilm_n \cdot I \cdot \Hilm_m
  \subseteq \Hilm_n \cdot \Hilm_m \cdot I
  \subseteq \Hilm_{n+m} \cdot I
  \]
  for all \(n,m\in\N\).
  Hence~\((S'_n)_{n\in\N}\)
  descends to a representation of~\(\Hilm|_{A/I}\).
  This representation is unitarily equivalent to the Fock
  representation.  Hence~\(\Hilm|_{A/I}\).
  is a partial product system.  Its Toeplitz \(\Cst\)\nb-algebra
  is isomorphic to \(\pi(\Toep)\)
  by Theorem~\ref{the:Toeplitz_Fock}.  And this is a quotient
  of~\(\Toep\) by a gauge-invariant ideal.
\end{proof}

\section{Gauge-Invariant ideals in the Toeplitz algebra}
\label{sec:gauge-invariant_ideals}

Let~\(B\)
be a \(\Cst\)\nb-algebra
with a continuous \(\T\)\nb-action
and let \(B_n\subseteq B\)
be the homogeneous subspaces.  Call an ideal \(J\idealin B_0\)
\emph{invariant} if \(J\cdot B_n \subseteq B_n\cdot J\)
for all \(n\in\Z\);
since \(B_n^* = B_{-n}\),
this is equivalent to \(J\cdot B_n = B_n\cdot J\)
only for \(n\in\N_{>0}\).
Given a \(\T\)\nb-invariant
ideal \(I\idealin B\),
its restriction \(I\cap B_0\)
is an invariant ideal in~\(B_0\).
Conversely, if \(J\idealin B_0\)
is an invariant ideal, then \(J\cdot B = B\cdot J\)
is a \(\T\)\nb-invariant
ideal in~\(B\).
It is well known that these two maps are isomorphisms inverse to each
other between the lattices of gauge-invariant ideals in~\(B\)
and of invariant ideals in~\(B_0\).

We are going to apply this general result to the Toeplitz
\(\Cst\)\nb-algebra
of a partial product system with its canonical \(\T\)\nb-action.
The gauge-invariant ideals of the
Toeplitz \(\Cst\)\nb-algebra
of an ordinary product system are described completely by
Katsura~\cite{Katsura:Ideal_structure_correspondences}.  Like Katsura,
we describe gauge-invariant ideals in a Toeplitz \(\Cst\)\nb-algebra
by a pair of ideals \(I\idealin J \idealin A\).
Here the ideal~\(I\)
is invariant, and any invariant ideal occurs.  We do not know, in
general, which ideals~\(J\)
are possible.  So our result is not as complete as Katsura's result
for product systems.

Throughout this section, we fix a \(\Cst\)\nb-algebra~\(A\)
and a partial product system
\(\Hilm= (A,\Hilm_n,\mu_{n,m})_{n,m\in\N}\).
Let~\(\Toep\)
be its Toeplitz \(\Cst\)\nb-algebra
and let \(\Toep_n\subseteq \Toep\)
for \(n\in\Z\)
be its homogeneous subspaces for the canonical \(\T\)\nb-action.
Let~\(\Hilm[F]\)
be~\(\bigoplus_{n=0}^\infty \Hilm_n\)
with the Fock representation~\((S_n)_{n\in\N}\)
of~\(\Hilm\).
To simplify notation, we view \(\Comp(\Hilm_m,\Hilm_n)\)
for \(m,n\in\N\)
as subspaces of~\(\Toep\),
that is, we drop the name~\(\Theta_{m,n}\)
for their canonical embeddings.  In particular, we view
\(\Hilm_n \cong \Comp(A,\Hilm_n)\)
and \(A=\Hilm_0\) as subspaces of~\(\Toep\).

\begin{lemma}
  \label{lem:kernel_A_invariant}
  If \(H\idealin\Toep\)
  is an ideal, then \(I\defeq H\cap A\idealin A\)
  is an invariant ideal, that is,
  \(I\cdot \Hilm_n \subseteq \Hilm_n\cdot I\)
  for all \(n\in\N\).
  Conversely, if~\(I\)
  is an invariant ideal in~\(A\),
  then the kernel~\(H\)
  of the canonical homomorphism
  \(\Toep\to\Toep(\Hilm|_{A/I})\)
  is a gauge-invariant ideal with \(I = H\cap A\).
  It is the minimal ideal \(H\idealin\Toep\)
  with \(I \subseteq H\cap A\).
\end{lemma}

\begin{proof}
  Let \(H\idealin\Toep\).
  There is a canonical representation~\((\omega_n)_{n\in\N}\)
  of~\(\Hilm\)
  in~\(\Toep/H\),
  and~\(I\)
  is its kernel.  Hence~\(I\)
  is invariant by Lemma~\ref{lem:faithful_A_Hilm_Comp}.  Conversely,
  let \(I\idealin A\)
  be an invariant ideal.  Then we may restrict our partial product
  system to the quotient~\(A/I\)
  as in Section~\ref{sec:restrict_quotient}.  This restriction is a
  partial product system~\(\Hilm|_{A/I}\),
  whose Toeplitz \(\Cst\)\nb-algebra
  \(\Toep(\Hilm|_{A/I})\)
  is a quotient of~\(\Toep\)
  by some gauge-invariant ideal~\(H\)
  by Proposition~\ref{pro:quotient_by_invariant_ideal}.  The map
  from~\(A/I\)
  to the Toeplitz \(\Cst\)\nb-algebra
  of \(\Hilm|_{A/I}\)
  is injective by Theorem~\ref{the:Toeplitz_gauge_fixed}.  So the
  intersection \(H\cap A\) is~\(I\).

  Let \(L\idealin \Toep\)
  be any ideal with \(I\subseteq L\cap A\).
  The canonical representation of~\(\Hilm\)
  in~\(\Toep\)
  induces a representation in~\(\Toep/L\).
  This representation kills~\(I\)
  and hence also \(\Hilm_n\cdot I\)
  for all \(n\in\N\).
  Therefore, it descends to a representation of~\(\Hilm|_{A/I}\).
  Hence the quotient map \(\Toep\prto\Toep/L\)
  factors through the Toeplitz \(\Cst\)\nb-algebra
  of~\(\Hilm|_{A/I}\).
  Thus~\(H\)
  above is minimal among the ideals \(L\idealin \Toep\)
  with \(I\subseteq L\cap A\).
\end{proof}

\begin{definition}
  \label{def:covariance_ideal}
  Let \((\omega_n)_{n\in\N}\)
  be a representation of the partial product system~\(\Hilm\)
  in a \(\Cst\)\nb-algebra~\(B\).
  Let \(\Hilm_{\ge n} B \subseteq B\)
  denote the closed linear span of \(\omega_m(\Hilm_m)\cdot B\)
  for all \(m\ge n\).
  The \emph{covariance ideal} of the representation is
  \[
  J(\omega_n) \defeq
  \setgiven{a\in A}{\omega_0(a)\cdot B\subseteq \Hilm_{\ge1} B}.
  \]
  We call~\((\omega_n)_{n\in\N}\)
  \emph{covariant} on an ideal \(J\idealin A\)
  if \(J\subseteq J(\omega_n)\).
\end{definition}

We are going to characterise a gauge-invariant ideal~\(H\)
by the pair~\((I,J)\)
of ideals in~\(A\),
where \(I\defeq H\cap A\)
and~\(J\)
is the covariance ideal of the canonical representation of~\(\Hilm\)
in~\(\Toep/H\).

For any ideal \(J\idealin A\),
there is a \(\Cst\)\nb-algebra
extension
\[
\Bound(\Hilm[F],\Hilm[F] J) \into
\Bound(\Hilm[F]) \prto \Bound(\Hilm[F]/\Hilm[F] J),
\]
which is also \(\T\)\nb-equivariant
(see the proof of Proposition~\ref{pro:quotient_by_invariant_ideal}).
Hence
\[
V(J) \defeq \Toep\cap \Bound(\Hilm[F],\Hilm[F]J)
\]
is a gauge-invariant ideal in~\(\Toep\)
for each ideal \(J\idealin A\).
It consists of those operators in~\(\Toep\)
that act by zero on \(\Hilm[F]/\Hilm[F]J\).
Let \(L^0(J)\subseteq\Bound(\Hilm[F]/\Hilm[F] J)\)
be the \(\T\)\nb-invariant
\Star{}subalgebra of finite block matrices:
\[
L^0(J) \defeq \setgiven{ x\in \Bound(\Hilm[F]/\Hilm[F] J)}
{\text{there is }N\in\N\text{ with }x_{n,m}=0\text{ for }n>N
\text{ or }m>N}.
\]
Here \(x_{n,m}\in\Bound(\Hilm_m/\Hilm_m J,\Hilm_n/\Hilm_n J)\)
denotes the \(n,m\)-matrix entry of~\(x\).  Let
\(L(J) \subseteq \Toep\) be the preimage of
\(\overline{L^0(J)}\subseteq \Bound(\Hilm[F]/\Hilm[F]J)\).
This is an ideal because \(\omega_n(x)\)
is a multiplier of~\(L^0(J)\)
for each \(x\in\Hilm_n\), \(n\in\N\).
The ideal~\(L(J)\)
is also \(\T\)\nb-invariant.

\begin{theorem}
  \label{the:ideal_Toeplitz}
  Let~\(H\)
  be a gauge-invariant ideal in the Toeplitz \(\Cst\)\nb-algebra
  of~\(\Hilm\).
  Let \(I\defeq H\cap A\)
  and let~\(J\)
  be its covariance ideal.  Then \(H = V(J) \cap L(I)\).
  In particular, the ideals \(I\)
  and~\(J\)
  determine~\(H\) uniquely.
\end{theorem}

The proof will use a couple of lemmas.  First we examine the
covariance ideal more closely.  The right ideal
\(\Hilm_{\ge1}\Toep\)
is easily seen to be the closure of
\(\sum_{n\ge0,m\ge1} \Comp(\Hilm_n,\Hilm_m)\).
Its intersection with~\(\Toep_0\)
is the closure of
\(\sum_{m\ge1} \Comp(\Hilm_m)\).
So the covariance ideal~\(J\)
for the canonical representation in~\(\Toep/H\)
is the set of all \(a\in A\)
that are identified with an element in the closure of
\(\sum_{m\ge1} \Comp(\Hilm_m)\)
in the quotient~\(\Toep/H\).
For instance, a Cuntz--Pimsner-like covariance condition
for~\(\Hilm_m\)
would identify \(a\sim\vartheta_0^m(a)\)
for certain elements \(a\in A\)
with \(\vartheta_0^m(a)\in\Comp(\Hilm_m)\).
Here \(\vartheta_j^k\colon \Comp(\Hilm_j) \to \Bound(\Hilm_k)\)
for \(0\le j \le k\)
is the map defined by the Fock representation, that is,
\(\vartheta_j^k(\ket{x}\bra{y})(z) = S_j(x) S_j(y)^* z\)
for all \(x,y\in\Hilm_j\), \(z\in\Hilm_k\).
The covariance ideal allows more complicated relations that identify
some elements of~\(A\)
with elements of the closure of
\(\sum_{m\ge1} \Comp(\Hilm_m)\).

If \(n,k\in\N\), then
\begin{equation}
  \label{eq:omega_shifts}
  \omega_n(\Hilm_n) \Hilm_{\ge k}B \subseteq \Hilm_{\ge n+k}B,\qquad
  \omega_n(\Hilm_n)^* \Hilm_{\ge n+k}B \subseteq \Hilm_{\ge k}B.
\end{equation}
the second property uses condition~\ref{simplifiability_5} in
Proposition~\ref{simplifiability}.  Hence the induced representations
\(\Theta_{n,m}\colon \Comp(\Hilm_n,\Hilm_m) \to B\) satisfy
\begin{equation}
  \label{eq:Theta_shifts}
  \Theta_{n,m}(\Comp(\Hilm_n,\Hilm_m)) \Hilm_{\ge k}B \subseteq
  \Hilm_{\ge \max\{m,k-n+m\}}B.
\end{equation}

\begin{lemma}
  \label{lem:covariant_on_Comp}
  Let \((\omega_n)_{n\in\N}\)
  be a representation of the partial product
  system~\(\Hilm\)
  in a \(\Cst\)\nb-algebra~\(B\) with covariance ideal~\(J\).  Then
  \[
  \Comp(\Hilm_n,\Hilm_m J) = \setgiven{T\in \Comp(\Hilm_n,\Hilm_m)}
  {\Theta_{n,m}(T)\cdot \Hilm_{\ge n}B \subseteq \Hilm_{\ge m+1} B}.
  \]
\end{lemma}

\begin{proof}
  Let \(T\in\Comp(\Hilm_n,\Hilm_m)\).
  We first assume \(T\in\Comp(\Hilm_n,\Hilm_m J)\).
  Then~\(T\)
  is in the closed linear span of operators of the form
  \(\ket{x} a \bra{y}\)
  with \(x\in\Hilm_m\),
  \(a\in J\),
  \(y\in\Hilm_n\).  Equation~\eqref{eq:omega_shifts} implies
  \[
  \Theta_{n,m}(\ket{x} a \bra{y})\Hilm_{\ge n}B
  \subseteq \omega_m(x) \omega_0(a) B
  \subseteq \omega_m(x) \Hilm_{\ge1}B
  \subseteq \Hilm_{\ge m+1}B.
  \]
  Hence \(\Theta_{n,m}(T)(\Hilm_{\ge n}B)\subseteq \Hilm_{\ge m+1}B\).

  Conversely, assume
  \(\Theta_{n,m}(T)(\Hilm_{\ge n}B)\subseteq \Hilm_{\ge m+1}B\).
  Let \(x\in\Hilm_n\), \(y\in\Hilm_m\).  Then
  \begin{multline*}
    \omega_0\bigl(\braket{y}{T(x)}\bigr) B
    = \omega_m(y)^* \omega_m(T(x)) B
    = \omega_m(y)^* \Theta_{n,m}(T) \omega_n(x) B
    \\ \subseteq \omega_m(y)^* \Theta_{n,m}(T) \Hilm_{\ge n} B
    \subseteq \omega_m(y)^* \Hilm_{\ge m+1} B
    \subseteq \Hilm_{\ge 1} B.
  \end{multline*}
  Thus \(\braket{y}{T(x)} \in J\).
  Since~\(y\)
  is arbitrary, this implies \(T(x) \in \Hilm_m\cdot J\).
  Since~\(x\)
  is arbitrary, this implies
  \(T \cdot \Comp(\Hilm_n) \subseteq \Comp(\Hilm_n,\Hilm_m \cdot J)\).
  Multiplying~\(T\)
  with an approximate unit in~\(\Comp(\Hilm_n)\),
  we get \(T\in\Comp(\Hilm_n,\Hilm_m \cdot J)\).
\end{proof}

\begin{lemma}
  \label{lem:kernel_covariance}
  Let \((\omega_n)_{n\in\N}\)
  be a representation of the partial product system~\(\Hilm\)
  in a \(\Cst\)\nb-algebra~\(B\).
  Let~\(J\)
  be its covariance ideal and let \(I \defeq \ker \omega_0\).
  Let \(\omega\colon \Toep\to B\)
  be the associated representation of the Toeplitz
  \(\Cst\)\nb-algebra~\(\Toep\).
  Let \(x_i\in\Comp(\Hilm_i)\)
  for \(i=0,\dotsc,N\)
  and \(X\defeq \sum_{i=1}^N x_i \in \Toep\).
  The following are equivalent:
  \begin{enumerate}
  \item \label{kernel_covariance_1}%
    \(\omega(X)=0\) in~\(B\);
  \item \label{kernel_covariance_2}%
    \(\sum\limits_{i=0}^\ell \vartheta_i^\ell(x_i)\in
    \Bound(\Hilm_\ell,\Hilm_\ell\cdot J)\)
    if \(\ell< N\)
    and
    \(\sum\limits_{i=0}^N
    \vartheta_i^\ell(x_i) \in \Bound(\Hilm_\ell,\Hilm_\ell\cdot I)\)
    if \(\ell\ge N\).
  \end{enumerate}
\end{lemma}

\begin{proof}
  We prove
  \ref{kernel_covariance_1}\(\Rightarrow\)\ref{kernel_covariance_2}
  and assume \(\omega(X)=0\).
  Let \(\ell\in\N\)
  and \(y\in\Hilm_\ell\).  Assume first that \(\ell\ge N\).  Then
  \[
  0 = \omega(X)\omega_\ell(y)
  = \omega_\ell\left(\sum_{i=0}^N \vartheta_i^\ell(x_i) y \right)
  \]
  by~\eqref{eq:Theta_multiplicative_0}.  This implies
  \(\sum_{i=0}^N \vartheta_i^\ell(x_i) y\in \Hilm_\ell\cdot I\)
  by Lemma~\ref{lem:faithful_A_Hilm_Comp} as asserted
  in~\ref{kernel_covariance_2}.  Now let \(\ell<N\).
  The same computation as above shows that
  \[
  \omega_\ell\left(\sum_{i=0}^\ell \vartheta_i^\ell(x_i) y \right)
  + \sum_{i=\ell+1}^N \omega(x_i) \omega_\ell(y) = 0.
  \]
  Since \(\omega(x_i) B \subseteq \Hilm_{\ge \ell+1} B\)
  for \(i\ge\ell+1\), this implies
  \[
  \omega_\ell\left(\sum_{i=0}^\ell \vartheta_i^\ell(x_i) y \right) B
  \subseteq \Hilm_{\ge \ell+1} B.
  \]
  Hence
  \(\sum_{i=0}^\ell \vartheta_i^\ell(x_i) y \in
  \Comp(\Hilm_0,\Hilm_\ell\cdot J) = \Hilm_\ell\cdot J\)
  by Lemma~\ref{lem:covariant_on_Comp}.  Since \(y\in\Hilm_\ell\)
  is arbitrary, this implies
  \(\sum_{i=0}^\ell \vartheta_i^\ell(x_i)\in
  \Bound(\Hilm_\ell,\Hilm_\ell\cdot J)\).
  This finishes the proof that
  \ref{kernel_covariance_1}\(\Longrightarrow\)\ref{kernel_covariance_2}.

  Now we prove, conversely, that~\ref{kernel_covariance_2}
  implies~\ref{kernel_covariance_1}.  If~\(X\)
  satisfies the conditions in~\ref{kernel_covariance_2}, then so
  does~\(X^*\).
  Hence we may replace~\(X\)
  by the two self-adjoint elements \(X+X^*\)
  and \(\ima^{-1}(X-X^*)\).
  So we may assume without loss of generality that~\(X\)
  is self-adjoint.  The assumption in~\ref{kernel_covariance_2} for
  \(\ell\ge N\)
  and Lemma~\ref{lem:faithful_A_Hilm_Comp} imply
  \(\omega(X) \omega_\ell(y)=0\)
  for all \(y\in\Hilm_\ell\),
  \(\ell\ge N\).
  Thus \(\omega(X)\)
  vanishes on~\(\Hilm_{\ge N}B=0\).
  Now let \(0\le\ell< N\)
  and \(y\in\Hilm_\ell\).
  Then \(\sum_{i=0}^\ell \vartheta_i^\ell(x_i) y \in \Hilm_\ell J\)
  by assumption.  Hence
  \begin{multline*}
    \omega(X) \omega_\ell(y) B
    = \sum_{i=0}^\ell \omega_\ell(\vartheta_i^\ell(x_i) y) B
    + \sum_{i=\ell+1}^N \omega(x_i) \omega_\ell(y) B
    \subseteq \omega_\ell(\Hilm_\ell J) B + \Hilm_{\ge \ell+1} B
    \\\subseteq \omega_\ell(\Hilm_\ell) \Hilm_{\ge1} B + \Hilm_{\ge \ell+1} B
    \subseteq \Hilm_{\ge \ell+1} B.
  \end{multline*}
  This implies
  \(\omega(X)\Hilm_{\ge \ell} B \subseteq \Hilm_{\ge \ell+1} B\)
  because
  \(\omega(X)\Hilm_{\ge \ell+1} B \subseteq \Hilm_{\ge \ell+1} B\)
  for any \(X\in\Toep_0\)
  by~\eqref{eq:Theta_shifts}.  Thus
  \(\omega(X)^N B\subseteq \Hilm_{\ge N} B\).
  Hence \(\omega(X)^{N+1}=0\).
  Then \(\omega(X)=0\)
  because~\(X\)
  is self-adjoint.  This finishes the proof
  that~\ref{kernel_covariance_2} implies~\ref{kernel_covariance_1}.
\end{proof}

\begin{proof}[Proof of Theorem~\textup{\ref{the:ideal_Toeplitz}}]
  Both \(H\)
  and \(V(J)\cap L(I)\)
  are gauge-invariant ideals in~\(\Toep\).
  Hence they are equal if and only if their intersections
  with~\(\Toep_0\)
  are equal.  And by Theorem~\ref{the:Toeplitz_gauge_fixed}, these
  intersections are equal if and only if the intersections
  with~\(\Toep_{0,N}\)
  are equal for all \(N\in\N\).
  Since these intersections are ideals, it suffices to prove that they
  have the same positive elements.  So let \(y\in\Toep_{0,N}\)
  be a positive element.  Describe~\(y\)
  by a block diagonal matrix on~\(\Hilm[F]\)
  with entries \(y_\ell \in \Bound(\Hilm_\ell)\)
  for \(\ell\in\N\).
  Here \(y_\ell = \sum_{i=0}^\ell \vartheta_i^\ell(x_i)\)
  in the notation of Lemma~\ref{lem:kernel_covariance}.  So
  Lemma~\ref{lem:kernel_covariance} says that \(y\in H\)
  if and only if \(y_\ell\in \Bound(\Hilm_\ell,\Hilm_\ell J)\)
  for \(0\le \ell < N\)
  and \(y_\ell\in \Bound(\Hilm_\ell,\Hilm_\ell I)\)
  for \(\ell \ge N\).
  Since \(I\subseteq J\),
  this implies \(y_\ell\in \Bound(\Hilm_\ell,\Hilm_\ell J)\)
  for all \(\ell\in\N\).
  And this is equivalent to
  \(y\in \Toep\cap \Bound(\Hilm[F],\Hilm[F] J) = V(J)\).
  Let~\(\bar{y}_\ell\)
  be the operator on \(\Hilm_\ell/\Hilm_\ell I\)
  induced by~\(y_\ell\).
  We have \(y_\ell\in \Bound(\Hilm_\ell,\Hilm_\ell I)\)
  if and only if \(\bar{y}_\ell = 0\).
  The condition \(y\in L(J)\)
  is equivalent to \(\lim_{\ell\to\infty} \norm{\bar{y}_\ell} = 0\).
  This clearly follows if \(\bar{y}_\ell = 0\)
  for \(\ell\ge N\).
  We claim the converse implication.  Since~\(y\)
  is positive, we may rewrite
  \(\lim_{\ell\to\infty} \norm{\bar{y}_\ell} = 0\)
  as follows: for each \(\varepsilon>0\),
  there is \(M\in\N\)
  such that \((\bar{y}_\ell-\varepsilon)_+ = 0\)
  for all \(\ell\ge M\);
  here \((\bar{y}_\ell-\varepsilon)_+\)
  means the positive part of \(\bar{y}_\ell-\varepsilon\).
  We may choose \(M\ge N\).
  Then \(y\in\Toep_{0,M}\),
  and Lemma~\ref{lem:kernel_covariance} and the conditions
  \((\bar{y}_\ell-\varepsilon)_+ = 0\)
  for \(\ell\ge M\)
  and \(y\in V(J)\)
  imply \((y-\varepsilon)_+ \in H\).
  Since \((y-\varepsilon)_+ \in \Toep_{0,N}\),
  Lemma~\ref{lem:kernel_covariance} implies
  \((\bar{y}_\ell-\varepsilon)_+ = 0\)
  already for \(\ell\ge N\).
  Since this holds for all \(\varepsilon>0\),
  we get \(\bar{y}_\ell=0\)
  for all \(\ell\ge N\).
\end{proof}

By Theorem~\ref{the:ideal_Toeplitz},
the lattice of gauge-invariant ideals in~\(\Toep\)
is isomorphic to the lattice of pairs of ideals \((I,J)\)
in~\(A\)
that occur as the kernel and covariance ideal for a gauge-invariant
ideal in~\(\Toep\).
If \((I,J)\)
comes from a gauge-invariant ideal~\(H\),
then \(H=V(J)\cap L(I)\)
by Theorem~\ref{the:ideal_Toeplitz}.  So the question is when
\(A\cap V(J)\cap L(I) = I\)
holds and the covariance ideal of \(V(J)\cap L(I)\idealin \Toep\)
is~\(J\).

We already know that~\(I\)
must be invariant and that any invariant ideal may occur.  And
\(I\subseteq J\)
is trivial.  Given an invariant ideal~\(I\),
we may form the quotient partial product system \(\Hilm_n/\Hilm_n I\).
Its Toeplitz \(\Cst\)\nb-algebra
is isomorphic to a quotient of~\(\Toep\)
by a gauge-invariant ideal by
Proposition~\ref{pro:quotient_by_invariant_ideal}.  Its kernel and
covariance ideal are \(I,I\)
by Theorem~\ref{the:Toeplitz_gauge_fixed}.  So
\[
\Toep(\Hilm|_{A/I}) \cong \Toep(\Hilm)\bigm/ (V(I) \cap L(I))
= \Toep(\Hilm)/ V(I)
\]
by Theorem~\ref{the:ideal_Toeplitz}.  We may replace the original
partial product system by \(\Hilm|_{A/I}\).
This reduces our problem to the case \(I=0\).

\begin{definition}
  \label{def:covariance_algebra}
  Let \(\Hilm = (A,\Hilm_n,\mu_{n,m})_{n,m\in\N}\)
  be a partial product system and let \(J\idealin A\) be an ideal.
  Let
  \[
  \CP(\Hilm,J) \defeq \Toep/(V(J) \cap L(0)).
  \]
  This quotient inherits a representation of~\(\Hilm\)
  and a \(\T\)\nb-action
  from~\(\Toep\)
  because \(V(J) \cap L(0)\)
  is a \(\T\)\nb-invariant
  ideal in~\(\Toep\).
  We call \(\CP(\Hilm,J)\)
  the \emph{\(J\)\nb-covariance
    algebra} if the canonical representation of~\(\Hilm\)
  in~\(\CP(\Hilm,J)\) has covariance ideal~\(J\).
  Let \(K\defeq \bigcap_{n=1}^\infty \ker \vartheta_0^n\).
  We call
  \[
  \CP_\mathrm{Pimsner}(\Hilm) \defeq \CP(\Hilm,A),\qquad
  \CP_\mathrm{Katsura}(\Hilm) \defeq \CP(\Hilm,K^\bot)
  \]
  the \emph{Pimsner algebra} and the \emph{Katsura algebra}
  of~\(\Hilm\), respectively.
\end{definition}

The name ``covariance algebra for~\(J\)''
is justified by the universal property in
Proposition~\ref{pro:CP_universal} below.
The Pimsner algebra is the quotient \(\Toep/L(0)\)
because \(V(A) = \Toep\).
Here \(L(0)\)
is defined as the norm-closure of the finite block matrices
in~\(\Toep\).
This is exactly how Pimsner defines his \(\Cst\)\nb-algebra for a
\(\Cst\)\nb-correspondence
in~\cite{Pimsner:Generalizing_Cuntz-Krieger}.
By Theorem~\ref{the:ideal_Toeplitz}, the Pimsner algebra is the
smallest quotient of~\(\Toep\)
that is defined by a covariance condition: any gauge-invariant
quotient that is
strictly smaller is of the form \(\Toep/(V(J)\cap L(I))\)
with a non-zero invariant ideal~\(I\).

\begin{proposition}
  \label{pro:CP_universal}
  Let \(J\idealin A\)
  be an ideal such that the canonical representation of~\(\Hilm\)
  in~\(\CP(\Hilm,J)\)
  has covariance ideal~\(J\).
  Then \Star{}homomorphisms \(\CP(\Hilm,J) \to B\)
  for a \(\Cst\)\nb-algebra~\(B\)
  are naturally in bijection with representations of the partial
  product system~\(\Hilm\)
  in~\(B\)
  that are covariant on~\(J\).
\end{proposition}

\begin{proof}
  Representations of~\(\CP(\Hilm,J)\)
  are in bijection with representations of~\(\Toep\)
  that kill \(V(J)\cap L(0) \idealin \Toep\).
  The universal property of the Toeplitz \(\Cst\)\nb-algebra gives a
  bijection between representations of~\(\Toep\)
  and representations of~\(\Hilm\)
  in~\(B\).
  Let \((\omega_n)_{n\in\N}\)
  be a representation of~\(\Hilm\)
  in~\(B\)
  and let \(\omega\colon \Toep\to B\)
  be the induced
  \Star{}homomorphism.
  Let \(I = \ker \omega_0 \idealin A\)
  and let~\(J(\omega_n)\)
  be the covariance ideal of~\((\omega_n)_{n\in\N}\).
  Then \(\ker \omega = V(J(\omega_n))\cap L(I)\)
  by Theorem~\ref{the:ideal_Toeplitz}.
  If~\((\omega_n)\)
  is covariant on~\(J\),
  that is, \(J\subseteq J(\omega_n)\),
  then \(V(J)\cap L(0) \subseteq V(J(\omega_n))\cap L(I)\)
  and hence \(\omega|_{V(J)\cap L(0)} = 0\).
  Then~\(\omega\)
  factors through~\(\CP(\Hilm,J)\).
  Conversely, assume that~\(\omega\)
  factors through~\(\CP(\Hilm,J)\).
  The covariance ideal for the representation of~\(\Hilm\)
  in~\(\CP(\Hilm,J)\)
  is~\(J\)
  by assumption.
  So~\(\omega\)
  is covariant on~\(J\).
\end{proof}

\begin{theorem}
  \label{the:covariance_ideal_not_too_large}
  The canonical \Star{}homomorphism \(A\to\CP(\Hilm,J)\)
  is faithful if and only if \(J\subseteq K^\bot\) for the ideal
  \[
  K \defeq \bigcap_{\ell=1}^\infty  \ker(\vartheta_0^\ell\colon A\to\Bound(\Hilm_\ell)).
  \]
\end{theorem}

\begin{proof}
  Assume first that~\(J\)
  is not contained in~\(K^\bot\).
  Then \(J\cap K\neq0\)
  and we may pick a non-zero element \(a\in J\cap K\).
  In the Fock representation, \(a\in A\)
  acts by the block diagonal operator with entries
  \(\vartheta_0^\ell(a)\in\Bound(\Hilm_\ell)\)
  for \(\ell\in\N\).
  By assumption, this belongs to~\(J\)
  for \(\ell=0\)
  and vanishes for \(\ell>0\).
  So \(a\in A\cap V(J) \cap L(0)\)
  becomes~\(0\)
  in~\(\CP(\Hilm,J)\).
  Conversely, assume that~\(J\)
  is contained in~\(K^\bot\).
  We must show that the map \(A \to \CP(\Hilm,J)\)
  is faithful.
  We prove a slightly more general claim, which will be needed
  below.  Namely, we treat \(x\in \Toep_{0,N}\)
  instead of \(a\in A\).
  We may write \(x= \sum_{i=0}^N x_i\)
  with \(x_i\in\Comp(\Hilm_i)\).
  In the Fock representation, this acts by the diagonal operator with
  entries
  \[
  y_\ell \defeq \sum_{i=0}^{\min\{\ell,N\}} \vartheta_i^\ell(x_i)
  \]
  for \(i\in\N\).
  We assume \(x\in V(J)\),
  that is, \(y_\ell\in \Bound(\Hilm_\ell,\Hilm_\ell J)\)
  for all \(\ell\in\N\).

  \begin{claim}
    \label{claim:511}
    An element \(x\in\Toep_{0,N} \cap V(J)\)
    satisfies \(x\in L(0)\)
    if and only if \(y_\ell=0\) for all \(\ell\ge N\).
  \end{claim}

  If \(N=0\),
  then \(\Toep_{0,0}=A\)
  and the claim says that \(x=0\)
  if \(x\in A \cap V(J) \cap L(0)\),
  which is what we have to prove.  So the following proof of the claim
  will also finish the proof of the proposition.
  The claim does not follow from Lemma~\ref{lem:kernel_covariance}
  because we do not yet know the vanishing and covariance ideals of
  the homomorphism \(A\to \CP(\Hilm,J)\).
  
  The proof of Theorem~\ref{the:ideal_Toeplitz} shows that elements
  for which there is \(M\in\N\)
  with \(y_\ell=0\)
  for all \(\ell>M\)
  are dense in \(\Toep_{0,N} \cap V(J) \cap L(0)\).
  To prove the claim, we must show that if there is \(M\in\N\)
  with \(y_\ell=0\)
  for all \(\ell>M\),
  then already \(y_\ell=0\)
  for all \(\ell\ge N\).
  Let~\(M\)
  be maximal with \(y_M\neq0\).
  We assume \(M>N\)
  in order to get to a contradiction.  Let \(\xi_1,\xi_2\in \Hilm_M\)
  and \(\eta\in\Hilm_i\)
  for some \(i>0\).
  Then \(\xi_2\cdot\eta\in\Hilm_{M+i}\)
  and so \(x\cdot \xi_2\cdot \eta = y_{M+i}(\xi_2\cdot\eta)=0\)
  because \(y_\ell=0\)
  for \(\ell>M\).
  (Here the product \(x\cdot \xi_2\cdot \eta\)
  takes place in~\(\Toep\),
  whereas \(y_{M+i}(\xi_2\cdot\eta)=0\)
  is an element of \(\Hilm_{M+i}\subseteq \Toep\).)
  We may also write \(x\cdot \xi_2\cdot\eta = y_M(\xi_2)\cdot\eta\).
  And
  \[
  0 = \xi_1^*\cdot x\cdot \xi_2\cdot\eta
  = \braket{\xi_1}{y_M(\xi_2)} \cdot \eta.
  \]
  The right hand side is the product of
  \(\braket{\xi_1}{y_M(\xi_2)} \in A\)
  with \(\eta\in\Hilm_i\).
  Since this vanishes for all \(\eta\in\Hilm_i\)
  for all \(i>0\),
  we get \(\braket{\xi_1}{y_M(\xi_2)} \in K\).
  Hence \(y_M(\xi_2) \in \Hilm_M\cdot K\).
  We also assumed \(y_M(\xi_2) \in \Hilm_M\cdot J\).
  So \(y_M(\xi_2) \in \Hilm_M\cdot (J\cap K)\).
  Hence \(y_M(\xi_2)=0\)
  because we assumed \(J\bot K\).
  Since \(\xi_2\in\Hilm_M\)
  is arbitrary, this implies \(y_M=0\).
  This is the desired contradiction, which proves the claim.
\end{proof}

Theorem~\ref{the:covariance_ideal_not_too_large} shows that
the covariance ideal of the representation of~\(\Hilm\)
in its Katsura algebra is maximal among the covariance ideals of
faithful representations.  In other words, the Katsura algebra is
the smallest quotient of~\(\Toep\)
for which the canonical map \(A\to\Toep\to \CP_\mathrm{Katsura}(\Hilm)\)
is injective.  This is the design principle of Katsura's
construction of a \(\Cst\)\nb-algebra
for a \(\Cst\)\nb-correspondence
in~\cite{Katsura:Cstar_correspondences}.

Product systems are easier than partial product systems because for
them the maps
\(\vartheta_m^n\colon \Comp(\Hilm_m) \to \Comp(\Hilm_n)\)
for \(m\le n\)
satisfy
\(\bar\vartheta_m^n\circ \vartheta_\ell^m = \vartheta_\ell^n\)
for all \(\ell\le m\le n\),
where~\(\bar\vartheta_m^n\)
is the canonical extension of~\(\vartheta_m^n\)
to \(\Bound(\Hilm_m)\).  Hence
\(\bigcap_{\ell=1}^\infty \ker \vartheta_0^\ell = \ker\vartheta_0^1\)
for product systems.

\begin{theorem}
  \label{the:product_system_covariance_ideal}
  Let~\(\Hilm\)
  be a product system and let \(H\idealin \Toep\)
  be a gauge-invariant ideal with \(A\cap H=0\).
  Then its covariance ideal is contained in
  \[
  J_{\max} \defeq (\vartheta_0^1)^{-1}(\Comp(\Hilm_1)) \cap
  (\ker \vartheta_0^1)^\bot.
  \]
  Any ideal \(J\idealin J_{\max}\)
  is the covariance ideal of a unique gauge-invariant ideal
  \(H\idealin \Toep\)
  with \(A\cap H=0\),
  namely, \(H= V(J) \cap L(0)\).
  The quotient \(\Toep/H\)
  is the \(J\)\nb-covariance algebra \(\CP(\Hilm,J)\).
  The Katsura algebra and the Pimsner algebra of~\(\Hilm\)
  are the \(\Cst\)\nb-algebras defined already by Katsura and
  Pimsner in this case.
\end{theorem}

\begin{proof}
  Let~\(J\)
  be the covariance ideal of~\(H\).
  Then \(H = V(J) \cap L(0)\)
  by Theorem~\ref{the:ideal_Toeplitz}.
  Theorem~\ref{the:covariance_ideal_not_too_large} implies
  \(J\subseteq(\ker \vartheta_0^1)^\bot\).
  We must prove \(\vartheta_0^1(J)\subseteq \Comp(\Hilm_1)\).

  Let \(N\ge1\)
  and let \(x_i\in \Comp(\Hilm_i)\)
  for \(i=0,\dotsc,N\)
  be such that \(\sum_{i=0}^N x_i \in \Toep_{0,N}\)
  belongs to \(H = V(J) \cap L(0)\).  For \(0\le \ell\le N\), define
  \[
  y_\ell \defeq \sum_{i=0}^\ell \vartheta_i^\ell(x_i) \in \Bound(\Hilm_\ell).
  \]

  \begin{claim}
    \(y_\ell\in \Comp(\Hilm_\ell)\)
    and \(\sum_{i=0}^{\ell-1} x_i + (x_\ell-y_\ell) \in H\)
    for \(1\le \ell\le N\).
  \end{claim}

  \begin{proof}
    We prove this recursively for \(\ell=N,N-1,N-2,\dotsc\).
    Claim~\ref{claim:511} implies \(y_N=0\).
    So the claim for \(\ell=N\)
    is our assumption \(\sum_{i=0}^N x_i \in H\).
    Assume the claim has been shown for some \(\ell>1\).
    We prove the claim for \(\ell-1\).
    Since~\(\Hilm\)
    is a global product system, the unique strictly continuous
    extension~\(\bar\vartheta_{\ell-1}^\ell\)
    of~\(\vartheta_{\ell-1}^\ell\)
    to \(\Bound(\Hilm_m)\)
    satisfies
    \(\bar\vartheta_{\ell-1}^\ell\circ \vartheta_i^{\ell-1} =
    \vartheta_i^\ell\) for \(0\le i \le \ell-1\).  So
    \[
    x_\ell-y_\ell
    = -\sum_{i=0}^{\ell-1} \bar\vartheta_{\ell-1}^\ell(\vartheta_i^{\ell-1}(x_i))
    = -\bar\vartheta_{\ell-1}^\ell(y_{\ell-1}).
    \]
    We first prove \(y_{\ell-1} \in \Comp(\Hilm_{\ell-1})\).
    This is clear if \(\ell=1\)
    because \(y_0 = x_0 \in A = \Comp(\Hilm_0)\).
    So let \(\ell>1\).
    Then
    \(\mu_{\ell-1,1}\colon \Hilm_{\ell-1} \otimes_A \Hilm_1 \to
    \Hilm_\ell\)
    is unitary.  Let \((u_\lambda)_{\lambda\in\Lambda}\)
    be an approximate unit for~\(\Comp(\Hilm_{\ell-1})\).
    Then \(\lim u_\lambda = 1\)
    in the strong topology on~\(\Hilm_{\ell-1}\).
    Since the net~\((u_\lambda)\)
    is self-adjoint and bounded, this implies strong convergence
    \(\lim u_\lambda \otimes_A \id_{\Hilm_1} = 1\)
    on \(\Hilm_{\ell-1} \otimes_A \Hilm_1\).
    This is equivalent to the strong convergence
    \(\lim \bar\vartheta_{\ell-1}^\ell(u_\lambda) = 1\)
    in~\(\Bound(\Hilm_\ell)\).
    Since the net \(\bar\vartheta_{\ell-1}^\ell(u_\lambda)\)
    is self-adjoint and bounded, this is equivalent to strict
    convergence.  The operator
    \(\bar\vartheta_{\ell-1}^\ell(y_{\ell-1})\)
    is compact by the claim for~\(\ell\).
    Hence the strict convergence of
    \(\bar\vartheta_{\ell-1}^\ell(u_\lambda)\)
    implies norm convergence
    \(\lim_\lambda \bar\vartheta_{\ell-1}^\ell(u_\lambda \cdot
    y_{\ell-1})= \bar\vartheta_{\ell-1}^\ell(y_{\ell-1})\).
    So
    \(\bar\vartheta_{\ell-1}^\ell(y_{\ell-1}) \in
    \bar\vartheta_{\ell-1}^\ell\bigl(\Comp(\Hilm_{\ell-1})\bigr)\).
    This is equivalent to
    \(y_{\ell-1} \in \Comp(\Hilm_{\ell-1}) + \ker
    \bar\vartheta_{\ell-1}^\ell\) because
    \[
    \bar\vartheta_{\ell-1}^\ell\bigl(\Comp(\Hilm_{\ell-1})\bigr) \cong
    \bigl(\Comp(\Hilm_{\ell-1}) + \ker \bar\vartheta_{\ell-1}^\ell\bigr)\bigm/\ker
    \bar\vartheta_{\ell-1}^\ell.
    \]
    The homomorphism~\(\bar\vartheta_{\ell-1}^\ell\)
    is injective on the ideal
    \(\Bound(\Hilm_{\ell-1},\Hilm_{\ell-1} J)\)
    by Lemma~\ref{lem:faithful_A_Hilm_Comp}
    and because \(J \cap \ker \vartheta_0^1 = 0\).
    Lemma~\ref{lem:kernel_covariance} implies
    \(y_{\ell-1} \in \Bound(\Hilm_{\ell-1},\Hilm_{\ell-1} J)\).
    This is orthogonal to \(\ker \bar\vartheta_{\ell-1}^\ell\).
    Hence
    \(y_{\ell-1} \in \Comp(\Hilm_{\ell-1}) + \ker
    \bar\vartheta_{\ell-1}^\ell\)
    implies \(y_{\ell-1} \in \Comp(\Hilm_{\ell-1})\).

    Now
    \(\sum_{i=0}^{\ell-2} x_i + (x_{\ell-1}-y_{\ell-1}) \in
    \Toep_{0,\ell-1}\)
    makes sense.  It belongs to~\(V(J)\)
    because \(\sum_{i=0}^{\ell-1} x_i \in V(J)\)
    by the claim for~\(\ell\)
    and \(y_{\ell-1} \in \Bound(\Hilm_{\ell-1},\Hilm_{\ell-1} J)\)
    by Lemma~\ref{lem:kernel_covariance}.  And
    \(\sum_{i=0}^{\ell-2} x_i + (x_{\ell-1}-y_{\ell-1})\)
    belongs to~\(L(0)\)
    because
    \(\sum_{i=0}^{\ell-2} \vartheta_i^{\ell-1}(x_i) +
    (x_{\ell-1}-y_{\ell-1}) = 0\)
    implies
    \(\sum_{i=0}^{\ell-2} \vartheta_i^k(x_i) +
    \vartheta_{\ell-1}^k(x_{\ell-1}-y_{\ell-1}) = 0\)
    for all \(k\ge \ell-1\)
    because
    \(\bar\vartheta_{\ell-1}^k\circ \vartheta_i^{\ell-1} =
    \vartheta_i^k\).  This finishes the proof of the claim.
  \end{proof}

  The proof of Theorem~\ref{the:ideal_Toeplitz} shows that the set of
  elements \(x_0\in A\)
  for which there are \(N\in\N\)
  and \(x_i\in \Comp(\Hilm_i)\)
  for \(i=1,\dotsc,N\)
  with \(\sum_{i=0}^N x_i \in H\)
  is dense in~\(J\).
  The claim for \(\ell=1\)
  says that \(\vartheta_0^1(x_0)\in\Comp(\Hilm_1)\)
  and \(x_0 - \vartheta_0^1(x_0) \in H\)
  for all such \(x_0\in A\).
  Hence \(\vartheta_0^1(J)\subseteq \Comp(\Hilm_1)\).
  This completes the proof that the covariance ideal of~\(H\)
  is contained in~\(J_{\max}\).
  In the Fock representation, \(x - \vartheta_0^1(x)\)
  for \(x\in A\)
  with \(\vartheta_0^1(x)\in \Comp(\Hilm_1)\)
  acts by \(x\cdot P_0\),
  where~\(P_0\)
  is the orthogonal projection onto~\(\Hilm_0\).
  So the claim also shows that \(x P_0 \in H\)
  for any \(x\in J\).
  Conversely, if \(x P_0 \in H\)
  for some \(x\in A\),
  then \(\vartheta_0^1(x)\in \Comp(\Hilm_1)\)
  follows, and so \(x\in J\).
  Thus the covariance ideal~\(J\)
  of~\(H\)
  is the set of all \(x\in A\)
  with \(x P_0 \in H\).
  The standard definition of a relative Cuntz--Pimsner
  in \cite{Muhly-Solel:Tensor}*{Definition~2.18}
  is to take the quotient of~\(\Toep\)
  by the ideal generated by \(J\cdot P_0\)
  for some ideal \(J\subseteq J_{\max}\).
  The argument above shows that this relative Cuntz--Pimsner has the
  covariance ideal~\(J\).
  Hence any \(J\idealin J_{\max}\)
  is a covariance ideal for some gauge-invariant ideal~\(H\)
  with \(H\cap A = 0\),
  and the resulting covariance algebra is the usual relative
  Cuntz--Pimsner algebra.  In particular, we get the
  \(\Cst\)\nb-algebra
  defined by Katsura~\cite{Katsura:Cstar_correspondences} for
  \(J=J_{\max}\).
  We have already observed after
  Definition~\ref{def:covariance_algebra}
  that~\(\CP_\mathrm{Pimsner}(\Hilm)\)
  is the \(\Cst\)\nb-algebra
  associated to the \(\Cst\)\nb-correspondence~\(\Hilm_1\)
  by Pimsner~\cite{Pimsner:Generalizing_Cuntz-Krieger}.
\end{proof}

\begin{remark}
  Theorem~\ref{the:product_system_covariance_ideal} says that the
  covariance ideal of any representation of a product system~\(\Hilm\)
  with \(K=0\)
  is contained in \((\vartheta_0^1)^{-1}(\Comp(\Hilm_1))\).
  This may fail if \(K\neq0\).
  Namely, it can happen that \(\vartheta_0^2(a)=0\)
  for some \(a\in A\)
  for which \(\vartheta_0^1(a)\)
  is not compact.  Then \(a \in L(0)\),
  although \(\vartheta_0^1(a)\)
  is not compact.  The Pimsner algebra in such a case is not a
  relative Cuntz--Pimsner algebra.
\end{remark}

Next we study the restriction of a Fell bundle
\(\mathcal{B} = (B_n)_{n\in\Z}\)
over~\(\Z\).
The multiplication and involution of the Fell bundle give a
\Star{}algebra structure on the direct sum \(\bigoplus_{n\in\Z} B_n\).
This \Star{}algebra has a maximal \(\Cst\)\nb-norm.
Its completion for this norm is the \emph{section \(\Cst\)\nb-algebra}
\(\Cst(\mathcal{B})\)
of the Fell bundle.  It carries a canonical \(\T\)\nb-action
where the subspaces \(B_n\subseteq \Cst(\mathcal{B})\)
for \(n\in\Z\)
are the homogeneous subspaces.  Since the maximal \(\Cst\)\nb-seminorm
on \(\bigoplus_{n\in\Z} B_n\)
is a \(\Cst\)\nb-norm,
the canonical maps \(B_n \to \Cst(\mathcal{B})\)
for \(n\in\Z\)
are all injective, even isometric.  In particular,
\(B_0 \injto \Cst(\mathcal{B})\).

\begin{theorem}
  \label{the:Fell_Cstar_covariance}
  Let~\(\mathcal{B}\)
  be a Fell bundle over~\(\Z\).
  The section \(\Cst\)\nb-algebra
  \(\Cst(\mathcal{B})\)
  is naturally isomorphic to the Katsura algebra of the
  restriction~\(\mathcal{B}|_\N\)
  of~\(\mathcal{B}\)
  to a partial product system over~\(\N\).
  The covariance ideal of the canonical representation
  of~\(\mathcal{B}|_\N\)
  in~\(\Cst(\mathcal{B})\)
  is the closed linear span of
  \(\sum_{n=1}^\infty \BRAKET{B_n}{B_n}\).
\end{theorem}

\begin{proof}
  The restriction~\(\mathcal{B}|_\N\)
  is a partial product system by Proposition~\ref{pro:pps_from_Fell}.
  Let~\(\Toep\)
  be its Toeplitz \(\Cst\)\nb-algebra.
  The canonical maps \(B_n \to \Cst(\mathcal{B})\)
  for \(n\in\N\)
  form a representation of~\(\mathcal{B}|_\N\).
  Hence they induce a \Star{}homomorphism
  \(\Toep\to\Cst(\mathcal{B})\).
  It is \(\T\)\nb-equivariant,
  and it is also surjective because its range contains the dense
  \Star{}subalgebra \(\bigoplus_{n\in\N} B_n\).
  Thus~\(\Cst(\mathcal{B})\)
  is the quotient of~\(\Toep\)
  by a gauge-invariant ideal~\(H\)
  in~\(\Toep\).
  We have \(H\cap A=0\)
  because \(A=B_0 \injto \Cst(\mathcal{B})\).
  Let~\(J\)
  denote the covariance ideal of the representation
  of~\(\mathcal{B}|_\N\)
  in~\(\Cst(\mathcal{B})\).
  It follows from Theorem~\ref{the:ideal_Toeplitz}
  that~\(\Cst(\mathcal{B})\)
  is the covariance algebra \(\CP(\mathcal{B}|_\N,J)\).
  Theorem~\ref{the:covariance_ideal_not_too_large} shows that
  \(H_{\max} \defeq L(0)\cap V(K^\bot)\)
  with \(K = \bigcap_{n=1}^\infty \ker(\vartheta_0^n)\).
  is the maximal gauge-invariant ideal \(H_{\max} \idealin \Toep\)
  with \(H_{\max} \cap A=0\).
  We claim that \(H = H_{\max}\).
  Since \(H\cap A=0\),
  the maximality of~\(H_{\max}\)
  gives \(H\subseteq H_{\max}\).
  For the converse inclusion, we show that \(L\cap A\neq0\)
  for any gauge-invariant ideal \(L\idealin\Toep\)
  with \(H\subsetneq L\).
  Since \(H\subseteq L\),
  the quotient~\(\Toep/L\)
  is a quotient of~\(\Cst(\mathcal{B})\)
  by a gauge-invariant ideal.  The gauge-invariant ideals in
  \(\Cst(\mathcal{B})\)
  are naturally in bijection with the \(\mathcal{B}\)\nb-invariant
  ideals in~\(B_0\).
  Since \(H\neq L\),
  the map \(B_0 \to \Cst(\mathcal{B})/(L/H)\)
  is not injective.  Thus \(L\cap A\neq0\)
  as claimed if \(H\subsetneq L\).
  So \(H=H_{\max}\)
  and \(\Cst(\mathcal{B})\)
  is the Katsura algebra of~\(\mathcal{B}|_\N\).

  Finally, we show that the covariance ideal~\(J\)
  of~\(\Cst(\mathcal{B})\)
  is the closed linear span of
  \(\sum_{n=1}^\infty \BRAKET{B_n}{B_n}\).
  The relation \(\BRAKET{x}{y} = x\cdot y^*\)
  for \(x,y\in B_n\),
  \(n\in\N\)
  holds in~\(\Cst(\mathcal{B})\).
  Hence \(\BRAKET{x}{y} \in B_{\ge1} \cdot \Cst(\mathcal{B})\)
  if \(n\ge1\).
  Thus the closed linear span of the ideals \(\BRAKET{B_n}{B_n}\)
  for \(n\in\N_{\ge1}\)
  is contained in the covariance ideal.  The right ideal
  \(\Hilm_{\ge1} \Cst(\mathcal{B})\)
  in~\(\Cst(\mathcal{B})\)
  is the closed linear span of \(B_\ell\cdot B_n\)
  for \(\ell\in\N_{\ge1}\),
  \(n\in\Z\).
  Since \(B_\ell\cdot B_\ell^*\cdot B_\ell = B_\ell\),
  it is equal to the closed linear span of
  \(B_\ell\cdot B_\ell^*\cdot B_n = \BRAKET{B_\ell}{B_\ell}\cdot
  B_n\).
  In particular, its gauge-invariant part is the closed linear span of
  \(\BRAKET{B_\ell}{B_\ell}\cdot B_0 = \BRAKET{B_\ell}{B_\ell}\).
  The covariance ideal must be contained in this ideal.  This
  inclusion and the reverse inclusion proved above give the assertion.
\end{proof}

Now we may answer the question when a partial product system of
Hilbert bimodules over~\(\N\)
is the restriction of a Fell bundle over~\(\Z\).
We mean here that the Fell bundle is such that it induces both the
multiplication maps and the inner products in the partial product
system.

\begin{theorem}
  \label{the:extend_pps_to_Fell}
  A weak partial product system
  \(\Hilm = (A,\Hilm_n,\mu_{n,m})_{n,m\in\N}\)
  is the restriction to~\(\N\)
  of a Fell bundle over~\(\Z\)
  if and only if each~\(\Hilm_n\)
  is a Hilbert \(A\)\nb-bimodule
  and
  \begin{align}
    \label{eq:Fell_bundle_extend_condition1}
    \BRAKET{\Hilm_n}{\Hilm_n}\cdot\Hilm_m
    &\subseteq \mu_{n,m-n}(\Hilm_n \otimes_A \Hilm_{m-n}),\\
    \label{eq:Fell_bundle_extend_condition2}
    \Hilm_m\cdot\braket{\Hilm_n}{\Hilm_n}
    &\subseteq \mu_{m-n,n}(\Hilm_{m-n} \otimes_A \Hilm_n)
  \end{align}
  as submodules of~\(\Hilm_m\)
  for all \(m,n\in\N\)
  with \(m\ge n\);
  here~\(\BRAKET{\blank}{\blank}\)
  denotes the left inner product.  The extension to a Fell bundle
  over~\(\Z\)
  is unique up to a canonical isomorphism.  The reverse inclusions to
  those in \eqref{eq:Fell_bundle_extend_condition1}
  and~\eqref{eq:Fell_bundle_extend_condition2} always hold, so that
  these inclusions are equivalent to equalities.
\end{theorem}

\begin{proof}
  The left action of the ideal \(\BRAKET{\Hilm_n}{\Hilm_n}\idealin A\)
  is nondegenerate on~\(\Hilm_n\)
  and hence also on \(\Hilm_n \otimes_A \Hilm_m\).
  Therefore,
  \(\BRAKET{\Hilm_n}{\Hilm_n}\cdot\Hilm_m \supseteq
  \mu_{n,m-n}(\Hilm_n \otimes_A \Hilm_{m-n})\)
  holds for any weak partial product system of Hilbert bimodules.  A
  dual proof shows the inclusion
  \(\Hilm_m\cdot\braket{\Hilm_n}{\Hilm_n} \supseteq
  \mu_{m-n,n}(\Hilm_{m-n} \otimes_A \Hilm_n)\).
  Thus the reverse inclusions to those in
  \eqref{eq:Fell_bundle_extend_condition1}
  and~\eqref{eq:Fell_bundle_extend_condition2} always hold.

  Now assume that~\(\Hilm\)
  is the restriction of a Fell bundle over~\(\Z\),
  which we also denote by \((\Hilm_n)_{n\in\Z}\).
  We may rewrite \(\BRAKET{x}{y} = x \cdot y^*\)
  and \(\braket{x}{y} = x^* \cdot y\)
  for all \(x,y\in\Hilm_n\), \(n\in\N\).  Hence
  \[
  \BRAKET{\Hilm_n}{\Hilm_n}\cdot\Hilm_m
  = \Hilm_n\cdot \Hilm_n^* \cdot \Hilm_m
  \subseteq \Hilm_n \cdot \Hilm_{m-n},
  \]
  which is equivalent to~\eqref{eq:Fell_bundle_extend_condition1}.  A
  dual argument gives~\eqref{eq:Fell_bundle_extend_condition2}.  Hence
  these two conditions are necessary for a weak partial product system
  to be the restriction of a Fell bundle over~\(\Z\).
  Now we assume these two conditions.  We are going to prove that our
  weak partial product system extends to a Fell bundle over~\(\Z\).

  First, we show that it is a partial product system.  Then we prove
  that \(J\defeq \sum_{n=1}^\infty \BRAKET{\Hilm_n}{\Hilm_n}\)
  is the covariance ideal for a gauge-invariant ideal
  \(H\idealin \Toep\)
  with \(H\cap A = 0\).
  So we get a faithful representation of the partial product
  system~\((\Hilm_n)_{n\in\N}\)
  in~\(\Toep/H\).
  Finally, we show that the images~\(B_n\)
  of~\(\Hilm_n\)
  in~\(\Toep/H\)
  for \(n\in\N\)
  and their adjoints \(B_{-n} \defeq B_n^*\)
  form a concrete Fell bundle over~\(\Z\),
  that is, \(B_n^* = B_{-n}\)
  for all \(n\in\Z\)
  and \(B_n\cdot B_m \subseteq B_{n+m}\)
  for all \(n,m\in\Z\).
  Since the representation of~\(\Hilm\)
  in~\(\Toep/H\)
  is faithful, this is the desired extension of~\((B_n)_{n\in\Z}\)
  to a Fell bundle over~\(\Z\).

  Let \(n,m\in\N_{>0}\)
  satisfy \(m\ge n\).
  The operator \(S_n(x)\colon \Hilm_{m-n} \to \Hilm_m\),
  \(y\mapsto x\otimes y\),
  is an adjointable map onto the Hilbert submodule
  \(\mu_{n,m-n}(\Hilm_n \otimes_A \Hilm_{m-n}) \subseteq \Hilm_m\).
  Equation~\eqref{eq:Fell_bundle_extend_condition1} identifies this
  with \(\BRAKET{\Hilm_n}{\Hilm_n}\cdot \Hilm_m\).
  Hence \(\BRAKET{y}{z}S_n(x) = S_n(\BRAKET{y}{z}x)\)
  for \(x,y,z\in\Hilm_n\)
  is an adjointable operator \(\Hilm_{m-n} \to \Hilm_m\).
  Since any element of~\(\Hilm_n\)
  may be written as \(\BRAKET{x}{y}z\)
  for suitable \(x,y,z\in\Hilm_n\),
  this shows that the operators
  \(S_n(x)\colon \Hilm_{m-n} \to \Hilm_m\) are adjointable.  And
  \[
  S_n(\Hilm_n)^* \Hilm_m
  = S_n(\Hilm_n)^* \BRAKET{\Hilm_n}{\Hilm_n} \Hilm_m
  = S_n(\Hilm_n)^* \mu_{n,m-n}(\Hilm_n \otimes_A\Hilm_{m-n})
  \subseteq \Hilm_{m-n}.
  \]
  Thus the Fock representation of our weak partial product system
  exists and is a representation.  Equivalently, \(\Hilm\)
  is a partial product system.

  Now let \(J\idealin A\)
  be the closure of \(\sum_{n=1}^\infty \BRAKET{\Hilm_n}{\Hilm_n}\).
  We claim that the gauge-invariant ideal
  \(V(J) \cap L(0)\idealin \Toep\)
  has zero intersection with~\(A\).
  If \(b\in \bigcap_{\ell=1}^\infty \ker\vartheta_0^\ell\),
  then \(b\cdot \Hilm_\ell=0\)
  and hence \(b\cdot\BRAKET{\Hilm_\ell}{\Hilm_\ell}=0\)
  for all \(\ell>0\).
  So \(b\cdot J=0\).
  Hence \(A\cap V(J) \cap L(0) = 0\)
  by Theorem~\ref{the:covariance_ideal_not_too_large}.  The
  covariance ideal of \(V(J) \cap L(0)\)
  is always contained in~\(J\).
  We prove the reverse inclusion.  Let
  \(a\in \BRAKET{\Hilm_n}{\Hilm_n}\)
  for \(n\ge1\).
  Then \(\vartheta_0^n(a)\)
  is a compact operator on~\(\Hilm_n\)
  because the left action on a Hilbert bimodule maps
  \(\BRAKET{\Hilm_n}{\Hilm_n}\)
  isomorphically onto the compact operators.  We claim that
  \(a - \vartheta_0^n(a) \in \Toep\)
  belongs to \(V(J) \cap L(0)\).
  Since
  \(\vartheta_0^n(a) \in \Hilm_{\ge n}\Toep \subseteq
  \Hilm_{\ge1}\Toep\),
  this implies that~\(a\) belongs to the covariance ideal.

  We prove the claim.  Let \(m\ge n\).
  Then the operators \(\vartheta_0^m(a)^*\)
  and \(\vartheta_n^m(\vartheta_0^n(a))^*\)
  on~\(\Hilm_n\)
  agree on the Hilbert submodule \(\Hilm_n\cdot\Hilm_{m-n}\).
  This is equal to \(\BRAKET{\Hilm_n}{\Hilm_n}\Hilm_m\)
  by assumption and hence contains
  \((\vartheta_0^m(a) - \vartheta_n^m(\vartheta_0^n(a))) b\)
  for all \(b\in\Hilm_m\).  So
  \[
  (\vartheta_0^m(a) - \vartheta_n^m(\vartheta_0^n(a)))^*
  \cdot (\vartheta_0^m(a) - \vartheta_n^m(\vartheta_0^n(a))) \cdot b
  = 0.
  \]
  This shows
  \(\vartheta_0^m(a)^* = \vartheta_n^m(\vartheta_0^n(a))^*\)
  for \(m\ge n\)
  as needed.  Now let \(\ell<n\).
  Then \(a-\vartheta_0^n(a)\in\Toep_0\)
  acts on the summand~\(\Hilm_\ell\)
  in the Fock representation by \(\vartheta_0^\ell(a)\).
  It maps~\(\Hilm_\ell\)
  into \(\BRAKET{\Hilm_n}{\Hilm_n}\cdot \Hilm_\ell\).
  We claim that this is contained in
  \(\Hilm_\ell \BRAKET{\Hilm_{n-\ell}}{\Hilm_{n-\ell}}\subseteq \Hilm_\ell J\).
  The proof of this claim will finish the proof that
  \(a - \vartheta_0^n(a) \in V(J) \cap L(0)\),
  which is all that remains to prove that the covariance ideal of
  \(V(J) \cap L(0)\) is~\(J\).

  Assumption~\eqref{eq:Fell_bundle_extend_condition1} implies that the
  range ideal of \(\Hilm_\ell \otimes_A \Hilm_{n-\ell}\)
  is
  \(\BRAKET{\Hilm_\ell}{\Hilm_\ell} \cap \BRAKET{\Hilm_n}{\Hilm_n}\).
  So the submodule
  \((\BRAKET{\Hilm_n}{\Hilm_n} \Hilm_\ell) \otimes_A \Hilm_{n-\ell}\)
  has the same range ideal and hence is equal to it.  Since its range
  ideal is equal to the range ideal of
  \(\BRAKET{\Hilm_n}{\Hilm_n} \Hilm_\ell\),
  the source ideal of \(\BRAKET{\Hilm_n}{\Hilm_n} \Hilm_\ell\)
  contains the range ideal \(\BRAKET{\Hilm_{n-\ell}}{\Hilm_{n-\ell}}\)
  of~\(\Hilm_{n-\ell}\).  Hence
  \[
  \BRAKET{\Hilm_n}{\Hilm_n} \Hilm_\ell
  = \BRAKET{\Hilm_n}{\Hilm_n} \Hilm_\ell
  \cdot \BRAKET{\Hilm_{n-\ell}}{\Hilm_{n-\ell}}
  \subseteq \Hilm_\ell \cdot \BRAKET{\Hilm_{n-\ell}}{\Hilm_{n-\ell}}.
  \]
  This proves the claim and shows that the covariance ideal of
  \(V(J) \cap L(0)\)
  is~\(J\).

  We have already seen that \(A \cap V(J) \cap L(0)= 0\).
  Hence the canonical representation
  \(\omega_n\colon \Hilm_n \to \Toep/(V(J) \cap L(0))\)
  is faithful.  We define \(B_n \defeq \omega_n(\Hilm_n)\)
  for \(n\in\N\)
  and \(B_{-n} \defeq B_n^* \subseteq\Toep/(V(J) \cap L(0))\).
  This satisfies \(B_n^* = B_{-n}\)
  for all \(n\in\Z\)
  by construction.  We claim that
  \begin{equation}
    \label{eq:concrete_Fell_bundle}
    B_n \cdot B_m \subseteq B_{n+m}
  \end{equation}
  holds for all \(n,m\in\Z\).
  Thus~\((B_n)_{n\in\Z}\)
  is a concrete Fell bundle.  It extends the partial product system of
  Hilbert bimodules~\(\Hilm\).

  If~\eqref{eq:concrete_Fell_bundle} holds for \((n,m)\in\Z^2\),
  then also
  \(B_{-m}\cdot B_{-n} = B_m^* \cdot B_n^* = (B_n\cdot B_m)^*
  \subseteq B_{n+m}^* = B_{-n-m}\),
  that is, \eqref{eq:concrete_Fell_bundle} holds for
  \((-m,-n)\in\Z^2\).
  Thus it suffices to prove~\eqref{eq:concrete_Fell_bundle} for
  those cases with \(n+m\ge0\).

  Let \(n,m\in\Z\)
  satisfy \(n+m\ge0\).
  Conditions
  \ref{def:weak_representation_1},
  \ref{def:weak_representation_2}
  and~\ref{def:weak_representation_4}
  in Definition~\ref{def:weak_representation}
  imply~\eqref{eq:concrete_Fell_bundle} if \(m\ge0\)
  and either \(n\ge 0\)
  or \(n=-m\)
  or \(-m<n<0\).
  This covers all \((n,m)\in\Z^2\) with \(n+m\ge0\) and \(n\ge0\).
  It remains to treat \((n,m)\in\Z^2\)
  with \(m<0\)
  and \(n+m\ge0\).
  So \(\ell\defeq -m\)
  satisfies \(0<\ell\le n\).
  We compute
  \[
  B_n B_{-\ell}
  = B_n B_\ell^* B_\ell B_\ell^*
  = B_n \braket{B_\ell}{B_\ell} B_\ell^*
  \subseteq B_{n-\ell} B_\ell B_\ell^*
  = B_{n-\ell} \BRAKET{B_\ell}{B_\ell}
  \subseteq B_{n-\ell};
  \]
  here the first step uses
  \(B_{-\ell} = B_\ell^* = B_\ell^* B_\ell B_\ell^*\);
  the second step uses the relation
  \(x^* y = \braket{x}{y}\)
  for all \(x,y\in B_\ell\)
  in the Toeplitz algebra; the third step
  uses~\eqref{eq:Fell_bundle_extend_condition2};
  the fourth step uses \(x y^* = \BRAKET{x}{y}\)
  for all \(x,y\in B_\ell\).
  Hence~\eqref{eq:concrete_Fell_bundle} holds for all \(n,m\in\Z\).
\end{proof}

\begin{proposition}
  \label{pro:graph_graded}
  As in Proposition~\textup{\ref{pro:graph_pps}}, let
  \(E=\bigsqcup_{n\in\N} E_n\)
  be the path category of a directed graph~\(\Gamma\),
  where generators in~\(\Gamma\)
  may have degrees different from~\(1\).
  The Katsura algebra of the partial product system defined
  by~\(E\)
  is the graph \(\Cst\)\nb-algebra of~\(\Gamma\),
  with the gauge action where the partial isometry associated to
  \(e\in E_n\) is \(n\)\nb-homogeneous for all \(n\in\N\).
\end{proposition}

\begin{proof}
  Let~\(\Hilm\)
  be the partial product system associated to~\(E\)
  by Proposition~\ref{pro:graph_pps} and let~\(\Toep\)
  be its Toeplitz \(\Cst\)\nb-algebra.
  The graph \(\Cst\)\nb-algebra~\(\Cst(\Gamma)\)
  receives a representation of the partial product system~\(\Hilm\)
  associated to~\(E\).
  This induces a homomorphism \(\Toep\to\Cst(\Gamma)\).
  It is surjective because all the partial isometries and projections
  generating~\(\Cst(\Gamma)\)
  belong to its range.  It is \(\T\)\nb-equivariant
  for the \(\T\)\nb-action
  specified in the proposition.  Hence \(\Cst(\Gamma)\)
  is a quotient of~\(\Toep\)
  by a gauge-invariant ideal.  The canonical map
  \(\Cont_0(E_0) \to \Cst(\Gamma)\)
  is injective.  Hence \(\Cst(\Gamma)\)
  is the covariance algebra for some ideal~\(J\)
  in~\(\Cont_0(E_0)\).
  The gauge-invariant ideals in \(\Cst(\Gamma)\)
  are described in
  \cite{Bates-Hong-Raeburn-Szymanski:Ideal_structure}*{Theorem~3.6}
  through a hereditary and saturated subset of~\(E_0\)
  and a set of breaking vertices.  As a consequence, the map
  \(\Cont_0(E_0) \to \Cst(\Gamma)/H\)
  for a non-zero gauge-invariant ideal \(H\idealin \Cst(\Gamma)\)
  is not injective.  Now it follows as in the proof of
  Theorem~\ref{the:Fell_Cstar_covariance} that \(\Cst(\Gamma)\)
  is the Katsura algebra of~\(\Hilm\).
\end{proof}

\begin{remark}
  All examples of covariance algebras treated above are quotients of
  the Toeplitz algebra by relations of the form
  \(a \sim \vartheta_0^\ell(a)\)
  for some \(a\in A\), \(\ell\in\N_{\ge1}\)
  with \(\vartheta_0^\ell(a)\in\Comp(\Hilm_\ell)\).
  These relations are direct analogues of the usual Cuntz--Pimsner
  covariance condition.  We should allow all \(\ell\ge1\)
  because \(\Hilm_1=0\)
  may happen.  And even for a global product system, it is possible
  to have covariance ideals that are larger than
  \((\vartheta_0^1)^{-1}(\Comp(\Hilm_1))\).
  Our formalism also allows relations of the form
  \(\sum_{i=0}^N x_i \sim 0\)
  for \(x_i\in\Comp(\Hilm_i)\).
  We do not know an example of a covariance algebra that cannot be
  obtained from relations of the simpler Cuntz--Pimsner form
  \(a \sim \vartheta_0^\ell(a)\).
\end{remark}

\end{document}